\documentclass[11pt,a4paper]{article}
\usepackage{amsmath,amssymb}
\usepackage{amsthm,stmaryrd}
\usepackage{bbm}
\usepackage{amsfonts,latexsym,bm}
\theoremstyle{thmstyleone}%
\newtheorem{theorem}{Theorem}
\newtheorem{lemma}{Lemma}
\theoremstyle{thmstyletwo}%
\newtheorem{remark}{Remark}%
\theoremstyle{thmstylethree}%
\allowdisplaybreaks
\begin{document}

\title{On some open problems in reliability theory}

\author{Geni Gupur\\
College of Mathematics and Systems Science,\\
 Xinjiang University, Huarui Street, Urumqi 830017,\\
  Xinjiang, P.R.China\\
E-mail: geni$\textcircled{a}$xju.edu.cn}

\date{}

\maketitle

\abstract{We study a stochastic scheduling on an unreliable machine with general up-times and general set-up times which is described by a group of partial differential equations with Dirac-delta functions in the boundary and initial conditions. In special case that the random processing rate of job $i,$  the random up-time rate of job $i$  and the random repair rate of job $i$ are constants, we determine the explicit expression of its time-dependent solution and give the asymptotic behavior of its time-dependent solution. Our result implies that $C_0-$semigroup theory is not suitable for this model. In general case, we determine the Laplace transform of its time-dependent solution. Next, we convert the model into an abstract Cauchy problem whose underlying operator is an evolution family. Finally, we leave some open problems.}

\section{Introduction}\label{sec1}

People are often concerned with the reliability of products they use and of the friends and associates with whom they interact daily in myriad of ways. The growth, recognition and definitization of the reliability function were given much impetus during the 1960s. During those processes, three main technical areas of reliability evolved: (1) reliability engineering; (2) operations analysis; (3) reliability mathematics. Each of these areas developed its own body of knowledge. This paper focuses on the mathematics of reliability. Throughout the history of reliability theory, large numbers of reliability problems were solved by using reliability models. There are several methods to establish such models. Among them the supplementary variable technique plays an important role. In 1955, Cox \cite{Cox1955} first put forward the supplementary variable technique and established the M/G/1 queueing model, readers may find dynamic analysis for the M/G/1 queueing model in \cite{GupurLiZhu,Gupur2011a}. In 1963, Gavier \cite{Gavier1963} first used the supplementary variable technique to study a reliability model. After that, other researchers widely applied this idea to study many reliability problems, see \cite{Gupur2011,Gupur2020,LiCao,LiShi,Linton1976} for instance. The reliability models that were established by the supplementary variable technique were described by partial differential equations with integral boundary conditions. According to our knowledge, these  partial differential equations with integral boundary conditions are split into two categories: partial differential equations without Dirac-delta functions in the boundary and initial conditions, partial differential equations with Dirac-delta functions in the boundary and initial conditions. We have a systematic method to study the partial differential equations without Dirac-delta functions in the boundary and initial conditions as follows \cite{Gupur2003,Gupur2014,Gupur2011,Gupur2016,Gupur2020,Gupur2022,Gupur2005,GupurWong,HajiRadl,HuXuYuZhu}: firstly we convert them into abstract Cauchy problems by choosing suitable state spaces and underlying operators; next by using $C_0-$semigroup theory we prove that they have unique positive time-dependent solutions; thirdly by studying spectral distribution on the imaginary axis we prove that their time-dependent solutions exponentially converge to their steady-state solutions if the number of the equations is finite, their time-dependent solutions strongly converge to their steady-state solutions if the number of the equations are infinite; finally by studying spectra of the underlying operators on the left half complex plane we give asymptotic expression of their time-dependent solutions if  the number of the equations is finite. It is pity that any results about dynamics of the reliability  models with Dirac-delta functions in the boundary and initial conditions have not been found in the literature. The main difficult point is the Dirac-delta function. Some researchers \cite{HajiKeyimYunus,Mijit2014} confused ``the reliability models with Dirac-delta functions in the boundary and initial conditions" with ``the reliability models without Dirac-delta functions in the boundary and initial conditions". \cite{HajiKeyimYunus,Mijit2014} applied $C_0-$semigroup theory to study well-posedness of two  reliability models with Dirac-delta functions and mistakenly believed ``they got the existence and uniqueness of the time-dependent solution". In this paper, we study the reliability model with Dirac-delta functions in the boundary and initial conditions in Li and Cao \cite{LiCao} and introduce our  results. In special case, by using the Laplace transform we deduce the explicit expression of the time-dependent solution of the model which contains Dirac-delta functions, thus we show that $C_0-$semigroup theory is not suitable for this model. In general case, we determine the Laplace transform of the time-dependent solution of the model, but we could not obtain the Laplace inverse transform and therefore we leave it as an open problem. Next, we convert the model into an abstract Cauchy problem whose underlying operator is an evolution family, thus we explain  difficulties of the model and leave several open problems in reliability theory.

In 1994, Li and Cao \cite{LiCao} discussed a stochastic scheduling on an unreliable machine with general up-times and general set-up times and established
the following mathematical model by using the supplementary variable technique:
\begin{align}
&\frac{\partial p_0(x,t)}{\partial t}+\frac{\partial p_0(x,t)}{\partial x}=-(\lambda(x)+\mu(x))p_0(x,t),\tag{1.1}\label{1.1}\\
&\frac{\partial p_1(x,t)}{\partial t}+\frac{\partial p_1(x,t)}{\partial x}=-\eta(x)p_1(x,t),\tag {1.2}\label{1.2}\\
&p_0(0,t)=\delta(t)+\int^{\infty}_0p_1(x,t)\eta(x)\mathrm{d}x,\tag {1.3}\label{1.3}\\
&p_1(0,t)=\int^{\infty}_0p_0(x,t)\lambda(x)\mathrm{d}x,\tag {1.4}\label{1.4}\\
&p_0(x,0)=\delta(x),\quad p_1(x,0)=0.\tag {1.5}\label{1.5}
\end{align}
Where $(x,t)\in [0,\infty)\times[0,\infty);$ $p_0(x,t)$ is the probability that at time $t,$ the machine is processing job $i$ with elapsed processing time $x;$
$p_1(x,t)$ represents the probability that at time $t,$ the machine has been subject to breakdown and is being repaired with elapsed repair time $x;$
$\mu(x)$ is the random processing rate of job $i$ at time $x$ satisfying
$$
\mu(x)\ge 0,\quad \int^{\infty}_0\mu(x)\mathrm{d}x=\infty.
$$
$\lambda(x)$ is the random up-time rate of job $i$ at time $x$ satisfying
$$
\lambda(x)\ge 0,\quad \int^{\infty}_0\lambda(x)\mathrm{d}x=\infty.
$$
$\eta(x)$ is the random repair rate of job $i$ at time $x$ satisfying
$$
\eta(x)\ge 0,\quad \int^{\infty}_0\eta(x)\mathrm{d}x=\infty.
$$
$\delta(\cdot)$ is the Dirac delta function satisfying
$$
\delta(x)=\begin{cases}
0 &x\not=0\\
\infty &x=0
\end{cases}
$$
By using the Laplace transform of the time-dependent solution of the above equation system (\ref{1.1})$\sim$(\ref{1.5}),
Li and Cao \cite{LiCao} considered optimal policy of preemptive repeat model and preemptive resume model.
So far, any other results of this model have not been found in the literature.

In this paper, we study the above system of equations (\ref{1.1})$\sim$(\ref{1.5}).  Firstly, we determine the explicit expression of the time-dependent solution of the above equations (\ref{1.1})$\sim$(\ref{1.5}) when $\lambda(x)=\lambda\;(\mbox{constant}),$ $\mu(x)=\mu\;(\mbox{constant}),$ $\eta(x)=\eta\;(\mbox{constant})$ and deduce its asymptotic behavior. Our result implies that $C_0-$semigroup theory is not suitable for (\ref{1.1})$\sim$(\ref{1.5}).
Next, we consider the case that  $\lambda(x),$ $\mu(x),$ $\eta(x)$ are non-constants and give the Laplace transform of the time-dependent solution of (\ref{1.1})$\sim$(\ref{1.5}). But we can not obtain the Laplace inverse transform, that is to say, we could not give the explicit expression of the time-dependent solution for general case. Lastly, by choosing a state space and an operator family we convert the equation system (\ref{1.1})$\sim$(\ref{1.5}) into an abstract Cauchy problem and therefore point out its differences from other Cauchy problems. We will finish with some open problems in reliability theory.

\section{Main Results}\label{sec2}

When $\lambda(x)=\lambda, \mu(x)=\mu, \eta(x)=\eta,$ the above system of  equations (\ref{1.1})$\sim$(\ref{1.5}) becomes
\begin{align}
&\frac{\partial p_0(x,t)}{\partial t}+\frac{\partial p_0(x,t)}{\partial x}=-(\lambda+\mu)p_0(x,t),\tag{2.1}\label{2.1}\\
&\frac{\partial p_1(x,t)}{\partial t}+\frac{\partial p_1(x,t)}{\partial x}=-\eta p_1(x,t),\tag {2.2}\label{2.2}\\
&p_0(0,t)=\delta(t)+\eta\int^{\infty}_0p_1(x,t)\mathrm{d}x,\tag {2.3}\label{2.3}\\
&p_1(0,t)=\lambda\int^{\infty}_0p_0(x,t)\mathrm{d}x,\tag {2.4}\label{2.4}\\
&p_0(x,0)=\delta(x),\quad p_1(x,0)=0.\tag {2.5}\label{2.5}
\end{align}

\begin{theorem}\label{geni}
If
\begin{align*}
&\lambda+\mu>\eta,\;\frac{\mathrm{d}}{\mathrm{d}t}\int^{\infty}_0p_k(x,t)e^{-sx}\mathrm{d}x=\int^{\infty}_0\frac{\partial p_k(x,t)}{\partial t}e^{-sx}\mathrm{d}x,\; k=0,1,
\end{align*}
then the time-dependent solution of the system (\ref{2.1})$\sim$(\ref{2.5}) is
\begin{align*}
p_0(x,t)&=2e^{-(\lambda+\mu)t}\delta(x-t)+\frac{\lambda\eta l_1}{z_1+\eta}e^{z_1t}e^{-(\lambda+\mu-\eta)x}\\
&\quad-\frac{\lambda\eta l_1}{z_1+\eta}e^{-\eta t}e^{-(\lambda+\mu-\eta)x}\\
&\quad+\frac{\lambda\eta l_2}{z_2+\eta}e^{z_2t}e^{-(\lambda+\mu-\eta)x}-\frac{\lambda\eta l_2}{z_2+\eta}e^{-\eta t}e^{-(\lambda+\mu-\eta)x}\\
&\quad-\frac{\lambda\eta l_1}{z_1+\eta}e^{z_1t}e^{-(\lambda+\mu-\eta)x}
+\frac{\lambda\eta l_1}{z_1+\eta}e^{z_1t}e^{-(\lambda+\mu+z_1)x}\\
&\quad+\frac{\lambda\eta l_1}{z_1+\eta}e^{-(\lambda+\mu)t}\int^x_0e^{-(\lambda+\mu-\eta)(x-z)}\delta(z-t)\mathrm{d}z\\
&\quad-\frac{\lambda\eta l_1}{z_1+\eta}e^{-(\lambda+\mu)t}\int^x_0e^{-(\lambda+\mu+z_1)(x-z)}\delta(z-t)\mathrm{d}z\\
&\quad-\frac{\lambda\eta l_2}{z_2+\eta}e^{z_2t}e^{-(\lambda+\mu-\eta)x}+\frac{\lambda\eta l_2}{z_2+\eta}e^{z_2t}e^{-(\lambda+\mu+z_2)x}\\
&\quad+\frac{\lambda\eta l_2}{z_2+\eta}e^{-(\lambda+\mu)t}\int^x_0e^{-(\lambda+\mu-\eta)(x-z)}\delta(z-t)\mathrm{d}z\\
&\quad-\frac{\lambda\eta l_2}{z_2+\eta}e^{-(\lambda+\mu)t}\int^x_0e^{-(\lambda+\mu+z_2)(x-z)}\delta(z-t)\mathrm{d}z,\\
p_1(x,t)&=2\lambda e^{(\lambda+\mu-\eta)x}e^{-(\lambda+\mu)t}-2\lambda e^{-\eta t}\int^x_0e^{(\lambda+\mu-\eta)(x-z)}\delta(z-t)\mathrm{d}z\\
&\quad+\frac{\lambda\eta h_1}{\lambda+\mu+z_1}e^{z_1t}e^{(\lambda+\mu-\eta)x}\\
&\quad-\frac{\lambda\eta h_1}{\lambda+\mu+z_1}e^{-(\lambda+\mu)t}e^{(\lambda+\mu-\eta)x}\\
&\quad+\frac{\lambda\eta h_2}{\lambda+\mu+z_2}e^{z_2t}e^{(\lambda+\mu-\eta)x}\\
&\quad-\frac{\lambda\eta h_2}{\lambda+\mu+z_2}e^{-(\lambda+\mu)t}e^{(\lambda+\mu-\eta)x}\\
&\quad-\frac{\lambda\eta h_1}{\lambda+\mu+z_1}e^{z_1t}e^{(\lambda+\mu-\eta)x}\\
&\quad+\frac{\lambda\eta h_1}{\lambda+\mu+z_1}e^{z_1t}e^{-(\eta+z_1)x}\\
&\quad+\frac{\lambda\eta h_1}{\lambda+\mu+z_1}e^{-\eta t}\int^x_0e^{(\lambda+\mu-\eta)(x-z)}\delta(z-t)\mathrm{d}z\\
&\quad-\frac{\lambda\eta h_1}{\lambda+\mu+z_1}e^{-\eta t}\int^x_0e^{-(\eta+z_1)(x-z)}\delta(z-t)\mathrm{d}z\\
&\quad-\frac{\lambda\eta h_2}{\lambda+\mu+z_2}e^{z_2t}e^{(\lambda+\mu-\eta)x}\\
&\quad-\frac{\lambda\eta h_2}{\lambda+\mu+z_2}e^{z_2t}e^{-(\eta+z_2)x}\\
&\quad+\frac{\lambda\eta h_2}{\lambda+\mu+z_2}e^{-\eta t}\int^x_0e^{(\lambda+\mu-\eta)(x-z)}\delta(z-t)\mathrm{d}z\\
&\quad-\frac{\lambda\eta h_2}{\lambda+\mu+z_2}e^{-\eta t}\int^x_0e^{-(\eta+z_2)(x-z)}\delta(z-t)\mathrm{d}z.
\end{align*}
where
\begin{align*}
z_1&=\frac{-(\lambda+\mu+\eta)+\sqrt{(\lambda+\mu+\eta)^2-4\mu\eta}}{2},\\
z_2&=\frac{-(\lambda+\mu+\eta)-\sqrt{(\lambda+\mu+\eta)^2-4\mu\eta}}{2},\\
l_1&=\frac{2(z_1-\eta)}{z_1-z_2},\quad l_2=-\frac{2(z_2-\eta)}{z_1-z_2},\\
h_1&=\frac{2\lambda}{z_1-z_2},\quad h_2=-\frac{2\lambda}{z_1-z_2}.
\end{align*}
Hence, by noting $\lambda>0, \mu>0, \eta>0, \Re z_1<0, \Re z_2<0$ it is easy to see that
$$
\lim_{t\to\infty}p_0(x,t)=0,\quad \lim_{t\to\infty}p_1(x,t)=0.
$$
\end{theorem}

\begin{proof}
By applying the Laplace transform with respect to $x$ to (\ref{2.1}) and (\ref{2.2}), and its properties
(see \cite{Arendt2001,Wazwaz}):
\begin{align}
\widehat{\frac{\mathrm{d}f}{\mathrm{d}x}}(s)&=s\hat{f}(s)-f(0),\tag{2.6}\label{2.6}\\
\widehat{(f+g)}&=\hat{f}+\hat{g},\;\widehat{af}=a\hat{f},\; a\in\mathbb{C},\tag {2.7}\label{2.7}
\end{align}
 we have
\begin{align}
&\frac{\partial\hat{p}_0(s,t)}{\partial t}+s\hat{p}_0(s,t)-p_0(0,t)=-(\lambda+\mu)\hat{p}_0(s,t)\Longrightarrow\notag\\
&\frac{\partial\hat{p}_0(s,t)}{\partial t}=-(\lambda+\mu+s)\hat{p}_0(s,t)+p_0(0,t),\tag {2.8}\label{2.8}\\
&\frac{\partial\hat{p}_1(s,t)}{\partial t}+s\hat{p}_1(s,t)-p_1(0,t)=-\eta\hat{p}_1(s,t)\Longrightarrow\notag\\
&\frac{\partial\hat{p}_1(s,t)}{\partial t}=-(\eta+s)\hat{p}_1(s,t)+p_1(0,t).\tag {2.9}\label{2.9}
\end{align}
By inserting (\ref{2.3}) into (\ref{2.8}) and (\ref{2.4}) into (\ref{2.9}) and using (see \cite{Arendt2001,Wazwaz})
$$
\int^{\infty}_0f(x)\mathrm{d}x=\hat{f}(0)
$$
we deduce
\begin{align}
\frac{\partial\hat{p}_0(s,t)}{\partial t}&=-(\lambda+\mu+s)\hat{p}_0(s,t)+\delta(t)+\eta\int^{\infty}_0p_1(x,t)\mathrm{d}x\notag\\
&=-(\lambda+\mu+s)\hat{p}_0(s,t)+\delta(t)+\eta\hat{p}_1(0,t)\Longrightarrow\notag\\
\hat{p}_0(s,t)&=\hat{p}_0(s,0)e^{-(\lambda+\mu+s)t}\notag\\
&\quad+e^{-(\lambda+\mu+s)t}\int^t_0[\delta(y)+\eta\hat{p}_1(0,y)]e^{(\lambda+\mu+s)y}\mathrm{d}y\notag\\
&=\hat{p}_0(s,0)e^{-(\lambda+\mu+s)t}+e^{-(\lambda+\mu+s)t}\int^t_0\delta(y)e^{(\lambda+\mu+s)y}\mathrm{d}y\notag\\
&\quad+\eta e^{-(\lambda+\mu+s)t}\int^t_0\hat{p}_1(0,y)e^{(\lambda+\mu+s)y}\mathrm{d}y.\tag {2.10}\label{2.10}\\
\frac{\partial\hat{p}_1(s,t)}{\partial t}&=-(\eta+s)\hat{p}_1(s,t)+\lambda\int^{\infty}_0p_0(x,t)\mathrm{d}x\notag\\
&=-(\eta+s)\hat{p}_1(s,t)+\lambda\hat{p}_0(0,t)\Longrightarrow\notag\\
\hat{p}_1(s,t)&=\hat{p}_1(s,0)e^{-(\eta+s)t}+\lambda e^{-(\eta+s)t}\int^t_0\hat{p}_0(0,y)e^{(\eta+s)y}\mathrm{d}y.\tag {2.11}\label{2.11}
\end{align}
(\ref{2.5}) gives
\begin{align}
\hat{p}_0(s,0)&=\int^{\infty}_0e^{-sx}p_0(x,0)\mathrm{d}x=\int^{\infty}_0e^{-sx}\delta(x)\mathrm{d}x=1,\tag {2.12}\label{2.12}\\
\hat{p}_1(s,0)&=\int^{\infty}_0e^{-sx}p_1(x,0)\mathrm{d}x=\int^{\infty}_00\mathrm{d}x=0.\tag {2.13}\label{2.13}
\end{align}
By combining (\ref{2.13}) with (\ref{2.11}) and (\ref{2.12}) with (\ref{2.10}) we obtain
\begin{align}
\hat{p}_0(s,t)&=e^{-(\lambda+\mu+s)t}+e^{-(\lambda+\mu+s)t}\int^t_0\delta(y)e^{(\lambda+\mu+s)y}\mathrm{d}y\notag\\
&\quad+\eta e^{-(\lambda+\mu+s)t}\int^t_0\hat{p}_1(0,y)e^{(\lambda+\mu+s)y}\mathrm{d}y,\tag {2.14}\label{2.14}\\
\hat{p}_1(s,t)&=\lambda e^{-(\eta+s)t}\int^t_0\hat{p}_0(0,y)e^{(\eta+s)y}\mathrm{d}y.\tag {2.15}\label{2.15}
\end{align}
By substituting (\ref{2.15}) into (\ref{2.14}) and using the Fubini theorem we deduce
\begin{align}
\hat{p}_0(s,t)&=e^{-(\lambda+\mu+s)t}+e^{-(\lambda+\mu+s)t}\int^t_0\delta(y)e^{(\lambda+\mu+s)y}\mathrm{d}y\notag\\
&\quad+\eta e^{-(\lambda+\mu+s)t}\int^t_0e^{(\lambda+\mu+s)y}\left[\lambda e^{-\eta y}\int^y_0\hat{p}_0(0,\tau)e^{\eta\tau}d\tau\right]\mathrm{d}y\notag\\
&=e^{-(\lambda+\mu+s)t}+e^{-(\lambda+\mu+s)t}\int^t_0\delta(y)e^{(\lambda+\mu+s)y}\mathrm{d}y\notag\\
&\quad+\lambda\eta e^{-(\lambda+\mu+s)t}\int^t_0e^{(\lambda+\mu+s-\eta)y}\int^y_0\hat{p}_0(0,\tau)e^{\eta\tau}d\tau\mathrm{d}y\notag\\
&=e^{-(\lambda+\mu+s)t}+e^{-(\lambda+\mu+s)t}\int^t_0\delta(y)e^{(\lambda+\mu+s)y}\mathrm{d}y\notag\\
&\quad+\lambda\eta e^{-(\lambda+\mu+s)t}\int^t_0\hat{p}_0(0,\tau)e^{\eta\tau}\int^t_{\tau}e^{(\lambda+\mu+s-\eta)y}\mathrm{d}y\mathrm{d}\tau\notag\\
&=e^{-(\lambda+\mu+s)t}+e^{-(\lambda+\mu+s)t}\int^t_0\delta(y)e^{(\lambda+\mu+s)y}\mathrm{d}y\notag\\
&\quad+\lambda\eta e^{-(\lambda+\mu+s)t}\int^t_0\hat{p}_0(0,\tau)e^{\eta\tau}
\frac{e^{(\lambda+\mu+s-\eta)y}}{\lambda+\mu+s-\eta}\Big\vert^{y=t}_{y=\tau}\mathrm{d}\tau\notag\\
&=e^{-(\lambda+\mu+s)t}+e^{-(\lambda+\mu+s)t}\int^t_0\delta(y)e^{(\lambda+\mu+s)y}\mathrm{d}y\notag\\
&\quad+\frac{\lambda\eta}{\lambda+\mu+s-\eta} e^{-(\lambda+\mu+s)t}\notag\\
&\quad\times\int^t_0\hat{p}_0(0,\tau)e^{\eta\tau}
\left(e^{(\lambda+\mu+s-\eta)t}-e^{(\lambda+\mu+s-\eta)\tau}\right)\mathrm{d}\tau\notag\\
&=e^{-(\lambda+\mu+s)t}+e^{-(\lambda+\mu+s)t}\int^t_0\delta(y)e^{(\lambda+\mu+s)y}\mathrm{d}y\notag\\
&\quad+\frac{\lambda\eta}{\lambda+\mu+s-\eta}e^{-(\lambda+\mu+s)t}\int^t_0\hat{p}_0(0,\tau)e^{\eta\tau}e^{(\lambda+\mu+s-\eta)t}\mathrm{d}\tau\notag\\
&\quad-\frac{\lambda\eta}{\lambda+\mu+s-\eta}e^{-(\lambda+\mu+s)t}\int^t_0\hat{p}_0(0,\tau)e^{(\lambda+\mu+s)\tau}\mathrm{d}\tau\notag\\
&=e^{-(\lambda+\mu+s)t}+e^{-(\lambda+\mu+s)t}\int^t_0\delta(y)e^{(\lambda+\mu+s)y}\mathrm{d}y\notag\\
&\quad+\frac{\lambda\eta}{\lambda+\mu+s-\eta}e^{-\eta t}\int^t_0\hat{p}_0(0,\tau)e^{\eta\tau}\mathrm{d}\tau\notag\\
&\quad-\frac{\lambda\eta}{\lambda+\mu+s-\eta}e^{-(\lambda+\mu+s)t}\int^t_0\hat{p}_0(0,\tau)e^{(\lambda+\mu+s)\tau}\mathrm{d}\tau.\tag {2.16}\label{2.16}
\end{align}
By combining (\ref{2.16}) with the Laplace transform and the Fubini theorem we derive when $\Re(\lambda+\mu+z)>0,\;\Re(\eta+z)>0$ (here $\Re(\eta+z)$ means real part of $\eta+z$), noting $\int^{\infty}_0\delta(y)e^{-zy}dy=1,$
\begin{align}
\hat{p}_0(0,t)&=e^{-(\lambda+\mu)t}+e^{-(\lambda+\mu)t}\int^t_0\delta(y)e^{(\lambda+\mu)y}\mathrm{d}y\notag\\
&\quad+\frac{\lambda\eta}{\lambda+\mu-\eta}e^{-\eta t}\int^t_0\hat{p}_0(0,\tau)e^{\eta\tau}\mathrm{d}\tau\notag\\
&\quad-\frac{\lambda\eta}{\lambda+\mu-\eta}e^{-(\lambda+\mu)t}\int^t_0\hat{p}_0(0,\tau)e^{(\lambda+\mu)\tau}\mathrm{d}\tau\notag\\
&\Longrightarrow\notag\\
\int^{\infty}_0e^{-tz}\hat{p}_0(0,t)\mathrm{d}t&=\int^{\infty}_0e^{-tz}e^{-(\lambda+\mu)t}\mathrm{d}t\notag\\
&\quad+\int^{\infty}_0e^{-tz}e^{-(\lambda+\mu)t}\int^t_0\delta(y)e^{(\lambda+\mu)y}\mathrm{d}y\mathrm{d}t\notag\\
&\quad+\frac{\lambda\eta}{\lambda+\mu-\eta}\int^{\infty}_0e^{-tz}e^{-\eta t}\int^t_0\hat{p}_0(0,\tau)e^{\eta\tau}\mathrm{d}\tau\mathrm{d}t\notag\\
&\quad-\frac{\lambda\eta}{\lambda+\mu-\eta}\int^{\infty}_0e^{-tz}e^{-(\lambda+\mu)t}\int^t_0\hat{p}_0(0,\tau)e^{(\lambda+\mu)\tau}\mathrm{d}\tau\mathrm{d}t\notag\\
&=\frac{-1}{\lambda+\mu+z}e^{-(\lambda+\mu+z)t}\Big\vert^{t=\infty}_{t=0}\notag\\
&\quad+\int^{\infty}_0\delta(y)e^{(\lambda+\mu)y}\int^{\infty}_ye^{-(\lambda+\mu+z)t}\mathrm{d}t\mathrm{d}y\notag\\
&\quad+\frac{\lambda\eta}{\lambda+\mu-\eta}\int^{\infty}_0\hat{p}_0(0,\tau)e^{\eta\tau}\int^{\infty}_{\tau}e^{-(\eta+z)t}\mathrm{d}t\mathrm{d}\tau\notag\\
&\quad-\frac{\lambda\eta}{\lambda+\mu-\eta}\int^{\infty}_0\hat{p}_0(0,\tau)e^{(\lambda+\mu)\tau}\int^{\infty}_{\tau}e^{-(\lambda+\mu+z)t}\mathrm{d}t\mathrm{d}\tau\notag\\
&=\frac{1}{\lambda+\mu+z}+\int^{\infty}_0\delta(y)e^{(\lambda+\mu)y}\frac{e^{-(\lambda+\mu+z)y}}{\lambda+\mu+z}\mathrm{d}y\notag\\
&\quad+\frac{\lambda\eta}{(\lambda+\mu-\eta)(\eta+z)}\int^{\infty}_0e^{-\tau z}\hat{p}_0(0,\tau)\mathrm{d}\tau\notag\\
&\quad-\frac{\lambda\eta}{(\lambda+\mu-\eta)(\lambda+\mu+z)}\int^{\infty}_0e^{-\tau z}\hat{p}_0(0,\tau)\mathrm{d}\tau\notag\\
&=\frac{1}{\lambda+\mu+z}+\frac{1}{\lambda+\mu+z}\int^{\infty}_0\delta(y)e^{-zy}\mathrm{d}y\notag\\
&\quad+\frac{\lambda\eta}{(\lambda+\mu-\eta)(\eta+z)}\int^{\infty}_0e^{-\tau z}\hat{p}_0(0,\tau)\mathrm{d}\tau\notag\\
&\quad-\frac{\lambda\eta}{(\lambda+\mu-\eta)(\lambda+\mu+z)}\int^{\infty}_0e^{-\tau z}\hat{p}_0(0,\tau)\mathrm{d}\tau\notag\\
&=\frac{2}{\lambda+\mu+z}\notag\\
&\quad+\frac{\lambda\eta}{(\lambda+\mu-\eta)(\eta+z)}\int^{\infty}_0e^{-\tau z}\hat{p}_0(0,\tau)\mathrm{d}\tau\notag\\
&\quad-\frac{\lambda\eta}{(\lambda+\mu-\eta)(\lambda+\mu+z)}\int^{\infty}_0e^{-\tau z}\hat{p}_0(0,\tau)\mathrm{d}\tau\notag\\
&\Longrightarrow\notag\\
&\Big[1-\frac{\lambda\eta}{(\lambda+\mu-\eta)(z+\eta)}\notag\\
&\quad+\frac{\lambda\eta}{(\lambda+\mu-\eta)(\lambda+\mu+z)}\Big]\int^{\infty}_0e^{-tz}\hat{p}_0(0,t)\mathrm{d}t\notag\\
&=\frac{2}{\lambda+\mu+z}\notag\\
&\Longrightarrow\notag\\
&\Big[1-\frac{\lambda\eta}{(\lambda+\mu+z)(z+\eta)}\Big]\int^{\infty}_0e^{-tz}\hat{p}_0(0,t)\mathrm{d}t=\frac{2}{\lambda+\mu+z}\notag\\
&\Longrightarrow\notag\\
&\frac{(\lambda+\mu+z)(z+\eta)-\lambda\eta}{(\lambda+\mu+z)(z+\eta)}\int^{\infty}_0e^{-tz}\hat{p}_0(0,t)\mathrm{d}t=\frac{2}{\lambda+\mu+z}\notag\\
&\Longrightarrow\notag\\
\int^{\infty}_0e^{-tz}\hat{p}_0(0,t)\mathrm{d}t&=\frac{2(z+\eta)}{(\lambda+\mu+z)(z+\eta)-\lambda\eta}\notag\\
&=\frac{2(z+\eta)}{z^2+(\lambda+\mu+\eta)z+\mu\eta}\notag\\
&=\frac{l_1}{z-z_1}+\frac{l_2}{z-z_2},\tag {2.17}\label{2.17}
\end{align}
where
\begin{align}
z_1&=\frac{-(\lambda+\mu+\eta)+\sqrt{(\lambda+\mu+\eta)^2-4\mu\eta}}{2},\tag{2.18}\label{2.18}\\
z_2&=\frac{-(\lambda+\mu+\eta)-\sqrt{(\lambda+\mu+\eta)^2-4\mu\eta}}{2},\tag {2.19}\label{2.19}\\
l_1&=\frac{2(z_1+\eta)}{z_1-z_2},\quad l_2=-\frac{2(z_2+\eta)}{z_1-z_2}.\tag {2.20}\label{2.20}
\end{align}
Since
\begin{align*}
&\int^{\infty}_0e^{-tz}e^{z_1t}\mathrm{d}t=\frac{1}{z-z_1},\;\Re(z-z_1)>0;\\
&\int^{\infty}_0e^{-tz}e^{z_2t}\mathrm{d}t=\frac{1}{z-z_2},\; \Re(z-z_2)>0,
\end{align*}
by (\ref{2.17}) we determine the Laplace inverse transform
\begin{align}
&\hat{p}_0(0,t)=l_1e^{z_1t}+l_2e^{z_2t}.\tag {2.21}\label{2.21}
\end{align}
By inserting (\ref{2.21}) into (\ref{2.16}) it gives
\begin{align}
\hat{p}_0(s,t)&=e^{-(\lambda+\mu+s)t}+e^{-(\lambda+\mu+s)t}\int^t_0\delta(y)e^{(\lambda+\mu+s)y}\mathrm{d}y\notag\\
&\quad+\frac{\lambda\eta}{\lambda+\mu+s-\eta}e^{-\eta t}\int^t_0\left[l_1e^{z_1\tau}+l_2e^{z_2\tau}\right]e^{\eta\tau}\mathrm{d}\tau\notag\\
&\quad-\frac{\lambda\eta}{\lambda+\mu+s-\eta}e^{-(\lambda+\mu+s)t}\int^t_0\left[l_1e^{z_1\tau}+l_2e^{z_2\tau}\right]e^{(\lambda+\mu+s)\tau}\mathrm{d}\tau\notag\\
&=e^{-(\lambda+\mu+s)t}+e^{-(\lambda+\mu+s)t}\int^t_0\delta(y)e^{(\lambda+\mu+s)y}\mathrm{d}y\notag\\
&\quad+\frac{\lambda\eta}{\lambda+\mu+s-\eta}e^{-\eta t}\int^t_0\left[l_1e^{(z_1+\eta)\tau}+l_2e^{(z_2+\eta)\tau}\right]\mathrm{d}\tau\notag\\
&\quad-\frac{\lambda\eta}{\lambda+\mu+s-\eta}e^{-(\lambda+\mu+s)t}\int^t_0\left[l_1e^{(\lambda+\mu+z_1+s)\tau}+l_2e^{(\lambda+\mu+z_2+s)\tau}\right]\mathrm{d}\tau\notag\\
&=e^{-(\lambda+\mu+s)t}+e^{-(\lambda+\mu+s)t}\int^t_0\delta(y)e^{(\lambda+\mu+s)y}\mathrm{d}y\notag\\
&\quad+\frac{\lambda\eta}{\lambda+\mu+s-\eta}e^{-\eta t}\left[\frac{l_1}{z_1+\eta}e^{(z_1+\eta)\tau}\Big\vert^{\tau=t}_{\tau=0}
+\frac{l_2}{z_2+\eta}e^{(z_2+\eta)\tau}\Big\vert^{\tau=t}_{\tau=0}\right]\notag\\
&\quad-\frac{\lambda\eta}{\lambda+\mu+s-\eta}e^{-(\lambda+\mu+s)t}\bigg[\frac{l_1}{\lambda+\mu+z_1+s}e^{(\lambda+\mu+z_1+s)\tau}\Big\vert^{\tau=t}_{\tau=0}\notag\\
&\quad+\frac{l_2}{\lambda+\mu+z_2+s}e^{(\lambda+\mu+z_2+s)\tau}\Big\vert^{\tau=t}_{\tau=0}\bigg]\notag\\
&=e^{-(\lambda+\mu+s)t}+e^{-(\lambda+\mu+s)t}\int^t_0\delta(y)e^{(\lambda+\mu+s)y}\mathrm{d}y\notag\\
&\quad+\frac{\lambda\eta}{\lambda+\mu+s-\eta}e^{-\eta t}\bigg\{\frac{l_1}{z_1+\eta}\left[e^{(z_1+\eta)t}-1\right]\notag\\
&\quad+\frac{l_2}{z_2+\eta}\left[e^{(z_2+\eta)t}-1\right]\bigg\}\notag\\
&\quad-\frac{\lambda\eta}{\lambda+\mu+s-\eta}e^{-(\lambda+\mu+s)t}\bigg\{\frac{l_1}{\lambda+\mu+z_1+s}\left[e^{(\lambda+\mu+z_1+s)t}-1\right]\notag\\
&\quad+\frac{l_2}{\lambda+\mu+z_2+s}\left[e^{(\lambda+\mu+z_2+s)t}-1\right]\bigg\}\notag\\
&=e^{-(\lambda+\mu+s)t}+e^{-(\lambda+\mu+s)t}\int^t_0\delta(y)e^{(\lambda+\mu+s)y}\mathrm{d}y\notag\\
&\quad+\frac{\lambda\eta l_1}{\lambda+\mu+s-\eta}\frac{1}{z_1+\eta}e^{z_1t}\notag\\
&\quad-\frac{\lambda\eta l_1}{\lambda+\mu+s-\eta}\frac{1}{z_1+\eta}e^{-\eta t}+\frac{\lambda\eta l_2}{\lambda+\mu+s-\eta}\frac{1}{z_2+\eta}e^{z_2t}\notag\\
&\quad-\frac{\lambda\eta l_2}{\lambda+\mu+s-\eta}\frac{1}{z_2+\eta}e^{-\eta t}-\frac{\lambda\eta l_1}{\lambda+\mu+s-\eta}\frac{1}{\lambda+\mu+z_1+s}e^{z_1t}\notag\\
&\quad+\frac{\lambda\eta l_1}{\lambda+\mu+s-\eta}\frac{1}{\lambda+\mu+z_1+s}e^{-(\lambda+\mu+s)t}\notag\\
&\quad-\frac{\lambda\eta l_2}{\lambda+\mu+s-\eta}\frac{1}{\lambda+\mu+z_2+s}e^{z_2t}\notag\\
&\quad+\frac{\lambda\eta l_2}{\lambda+\mu+s-\eta}\frac{1}{\lambda+\mu+z_2+s}e^{-(\lambda+\mu+s)t}\notag\\
&=e^{-(\lambda+\mu+s)t}+e^{-(\lambda+\mu+s)t}\int^t_0\delta(y)e^{(\lambda+\mu+s)y}\mathrm{d}y\notag\\
&\quad+\frac{\lambda\eta l_1}{z_1+\eta}e^{z_1t}\frac{1}{\lambda+\mu-\eta+s}\notag\\
&\quad-\frac{\lambda\eta l_1}{z_1+\eta}e^{-\eta t}\frac{1}{\lambda+\mu-\eta+s}+\frac{\lambda\eta l_2}{z_2+\eta}e^{z_2t}\frac{1}{\lambda+\mu-\eta+s}\notag\\
&\quad-\frac{\lambda\eta l_2}{z_2+\eta}e^{-\eta t}\frac{1}{\lambda+\mu-\eta+s}\notag\\
&\quad-\frac{\lambda\eta l_1}{z_1+\eta}\left(\frac{1}{\lambda+\mu-\eta+s}-\frac{1}{\lambda+\mu+z_1+s}\right)e^{z_1t}\notag\\
&\quad+\frac{\lambda\eta l_1}{z_1+\eta}\left(\frac{1}{\lambda+\mu-\eta+s}-\frac{1}{\lambda+\mu+z_1+s}\right)e^{-(\lambda+\mu+s)t}\notag\\
&\quad-\frac{\lambda\eta l_2}{z_2+\eta}\left(\frac{1}{\lambda+\mu-\eta+s}-\frac{1}{\lambda+\mu+z_2+s}\right)e^{z_2t}\notag\\
&\quad+\frac{\lambda\eta l_2}{z_2+\eta}\left(\frac{1}{\lambda+\mu-\eta+s}-\frac{1}{\lambda+\mu+z_2+s}\right)e^{-(\lambda+\mu+s)t}\notag\\
&=e^{-(\lambda+\mu)t}e^{-st}+e^{-(\lambda+\mu+s)t}\int^t_0\delta(y)e^{(\lambda+\mu+s)y}dy\notag\\
&\quad+\frac{\lambda\eta l_1}{z_1+\eta}e^{z_1t}\frac{1}{\lambda+\mu-\eta+s}\notag\\
&\quad-\frac{\lambda\eta l_1}{z_1+\eta}e^{-\eta t}\frac{1}{\lambda+\mu-\eta+s}+\frac{\lambda\eta l_2}{z_2+\eta}e^{z_2t}\frac{1}{\lambda+\mu-\eta+s}\notag\\
&\quad-\frac{\lambda\eta l_2}{z_2+\eta}e^{-\eta t}\frac{1}{\lambda+\mu-\eta+s}\notag\\
&\quad-\frac{\lambda\eta l_1}{z_1+\eta}e^{z_1t}\frac{1}{\lambda+\mu-\eta+s}+\frac{\lambda\eta l_1}{z_1+\eta}e^{z_1t}\frac{1}{\lambda+\mu+z_1+s}\notag\\
&\quad+\frac{\lambda\eta l_1}{z_1+\eta}e^{-(\lambda+\mu)t}\frac{1}{\lambda+\mu-\eta+s}e^{-st}\notag\\
&\quad-\frac{\lambda\eta l_1}{z_1+\eta}e^{-(\lambda+\mu)t}
\frac{1}{\lambda+\mu+z_1+s}e^{-st}\notag\\
&\quad-\frac{\lambda\eta l_2}{z_2+\eta}e^{z_2t}\frac{1}{\lambda+\mu-\eta+s}+\frac{\lambda\eta l_2}{z_2+\eta}e^{z_2t}\frac{1}{\lambda+\mu+z_2+s}\notag\\
&\quad+\frac{\lambda\eta l_2}{z_2+\eta}e^{-(\lambda+\mu)t}\frac{1}{\lambda+\mu-\eta+s}e^{-st}\notag\\
&\quad-\frac{\lambda\eta l_2}{z_2+\eta}e^{-(\lambda+\mu)t}\frac{1}{\lambda+\mu+z_2+s}e^{-st}.\tag {2.22}\label{2.22}
\end{align}
By combining (\ref{2.22}) with
\begin{align*}
&\int^{\infty}_0e^{-sx}\delta(x-t)\mathrm{d}x=e^{-st},\;\int^{\infty}_0e^{-sx}e^{-ux}\mathrm{d}x=\frac{1}{u+s},\;\Re(u+s)>0,\tag {2.23}\label{2.23}\\
&\int^{\infty}_0e^{-tz}e^{-(\lambda+\mu+s)t}\int^t_0\delta(y)e^{(\lambda+\mu+s)y}\mathrm{d}y\mathrm{d}t\\
&=\int^{\infty}_0\delta(y)e^{(\lambda+\mu+s)y}\int^{\infty}_ye^{-(\lambda+\mu+s+z)t}\mathrm{d}t\mathrm{d}y\\
&=\frac{1}{\lambda+\mu+s+z}\int^{\infty}_0\delta(y)e^{-zy}\mathrm{d}y=\frac{1}{\lambda+\mu+s+z}\\
&\Rightarrow\\
&e^{-(\lambda+\mu+s)t}\int^t_0\delta(y)e^{(\lambda+\mu+s)y}\mathrm{d}y=e^{-(\lambda+\mu+s)t},\;\Re(\lambda+\mu+s+z)>0,\tag {2.24}\label{2.24}\\
&\int^{\infty}_0e^{-sx}\left(\int^x_0e^{-u(x-z)}\delta(z-t)\mathrm{d}z\right)\mathrm{d}x\\
&=\int^{\infty}_0\int^{\infty}_ze^{-sx}e^{-u(x-z)}\delta(z-t)\mathrm{d}x\mathrm{d}z\\
&=\int^{\infty}_0e^{uz}\delta(z-t)\int^{\infty}_ze^{-sx}e^{-ux}\mathrm{d}x\mathrm{d}z\\
&=\int^{\infty}_0e^{uz}\delta(z-t)\frac{1}{u+s}e^{-(u+s)z}\mathrm{d}z\\
&=\frac{1}{u+s}\int^{\infty}_0e^{-sz}\delta(z-t)\mathrm{d}z\\
&=\frac{1}{u+s}e^{-st},\; \Re(u+s)>0.\tag {2.25}\label{2.25}
\end{align*}
we determine the Laplace inverse transform
\begin{align*}
p_0(x,t)&=2e^{-(\lambda+\mu)t}\delta(x-t)+\frac{\lambda\eta l_1}{z_1+\eta}e^{z_1t}e^{-(\lambda+\mu-\eta)x}\\
&\quad-\frac{\lambda\eta l_1}{z_1+\eta}e^{-\eta t}e^{-(\lambda+\mu-\eta)x}\\
&\quad+\frac{\lambda\eta l_2}{z_2+\eta}e^{z_2t}e^{-(\lambda+\mu-\eta)x}-\frac{\lambda\eta l_2}{z_2+\eta}e^{-\eta t}e^{-(\lambda+\mu-\eta)x}\\
&\quad-\frac{\lambda\eta l_1}{z_1+\eta}e^{z_1t}e^{-(\lambda+\mu-\eta)x}
+\frac{\lambda\eta l_1}{z_1+\eta}e^{z_1t}e^{-(\lambda+\mu+z_1)x}\\
&\quad+\frac{\lambda\eta l_1}{z_1+\eta}e^{-(\lambda+\mu)t}\int^x_0e^{-(\lambda+\mu-\eta)(x-z)}\delta(z-t)\mathrm{d}z\\
&\quad-\frac{\lambda\eta l_1}{z_1+\eta}e^{-(\lambda+\mu)t}\int^x_0e^{-(\lambda+\mu+z_1)(x-z)}\delta(z-t)\mathrm{d}z\\
&\quad-\frac{\lambda\eta l_2}{z_2+\eta}e^{z_2t}e^{-(\lambda+\mu-\eta)x}+\frac{\lambda\eta l_2}{z_2+\eta}e^{z_2t}e^{-(\lambda+\mu+z_2)x}\\
&\quad+\frac{\lambda\eta l_2}{z_2+\eta}e^{-(\lambda+\mu)t}\int^x_0e^{-(\lambda+\mu-\eta)(x-z)}\delta(z-t)\mathrm{d}z\\
&\quad-\frac{\lambda\eta l_2}{z_2+\eta}e^{-(\lambda+\mu)t}\int^x_0e^{-(\lambda+\mu+z_2)(x-z)}\delta(z-t)\mathrm{d}z.\tag {2.26}\label{2.26}
\end{align*}
By inserting (\ref{2.14}) into (\ref{2.15}) and using (\ref{2.24}) and the Fubini theorem we derive
\begin{align*}
\hat{p}_1(s,t)&=\lambda e^{-(\eta+s)t}\int^t_0e^{(\eta+s)y}\Big[2e^{-(\lambda+\mu)y}\\
&\quad+\eta e^{-(\lambda+\mu)y}\int^y_0\hat{p}_1(0,\tau)e^{(\lambda+\mu)\tau}\mathrm{d}\tau\Big]\mathrm{d}y\\
&=2\lambda e^{-(\eta+s)t}\int^t_0e^{(\eta+s-\lambda-\mu)y}\mathrm{d}y\\
&\quad+\lambda\eta e^{-(\eta+s)t}\int^t_0e^{(\eta+s-\lambda-\mu)y}\int^y_0\hat{p}_1(0,\tau)e^{(\lambda+\mu)\tau}
\mathrm{d}\tau\mathrm{d}y\\
&=2\lambda e^{-(\eta+s)t}\frac{e^{(\eta+s-\lambda-\mu)y}}{\eta+s-\lambda-\mu}\Big\vert^{y=t}_{y=0}\\
&\quad+\lambda\eta e^{-(\eta+s)t}\int^t_0\hat{p}_1(0,\tau)e^{(\lambda+\mu)\tau}
\int^t_{\tau}e^{(\eta+s-\lambda-\mu)y}\mathrm{d}y\mathrm{d}\tau\\
&=\frac{2\lambda}{\eta+s-\lambda-\mu}e^{-(\eta+s)t}\left[e^{(\eta+s-\lambda-\mu)t}-1\right]\\
&\quad+\lambda\eta e^{-(\eta+s)t}\int^t_0\hat{p}_1(0,\tau)e^{(\lambda+\mu)\tau}\frac{e^{(\eta+s-\lambda-\mu)y}}{\eta+s-\lambda-\mu}\Big\vert^{y=t}_{y=\tau}\mathrm{d}\tau\\
&=\frac{2\lambda}{\eta+s-\lambda-\mu}e^{-(\lambda+\mu)t}-\frac{2\lambda}{\eta+s-\lambda-\mu}e^{-(\eta+s)t}\\
&\quad+\frac{\lambda\eta}{\eta+s-\lambda-\mu} e^{-(\eta+s)t}\\
&\quad\times\int^t_0\hat{p}_1(0,\tau)e^{(\lambda+\mu)\tau}
\left[e^{(\eta+s-\lambda-\mu)t}-e^{(\eta+s-\lambda-\mu)\tau}\right]\mathrm{d}\tau\\
&=\frac{2\lambda}{\eta+s-\lambda-\mu}e^{-(\lambda+\mu)t}-\frac{2\lambda}{\eta+s-\lambda-\mu}e^{-(\eta+s)t}\\
&\quad+\frac{\lambda\eta}{\eta+s-\lambda-\mu}e^{-(\eta+s)t}e^{(\eta+s-\lambda-\mu)t}\int^t_0\hat{p}_1(0,\tau)e^{(\lambda+\mu)\tau}\mathrm{d}\tau\\
&\quad-\frac{\lambda\eta}{\eta+s-\lambda-\mu}e^{-(\eta+s)t}\int^t_0\hat{p}_1(0,\tau)e^{(\eta+s)\tau}\mathrm{d}\tau\\
&=\frac{2\lambda}{\eta+s-\lambda-\mu}e^{-(\lambda+\mu)t}-\frac{2\lambda}{\eta+s-\lambda-\mu}e^{-(\eta+s)t}\\
&\quad+\frac{\lambda\eta}{\eta+s-\lambda-\mu}e^{-(\lambda+\mu)t}\int^t_0\hat{p}_1(0,\tau)e^{(\lambda+\mu)\tau}\mathrm{d}\tau\\
&\quad-\frac{\lambda\eta}{\eta+s-\lambda-\mu}e^{-(\eta+s)t}\int^t_0\hat{p_1}(0,\tau)e^{(\eta+s)\tau}\mathrm{d}\tau.\tag{2.27}\label{2.27}
\end{align*}
By applying the Laplace transform with respect to $t$ to (\ref{2.27}) and the Fubini theorem  it gives for $\Re z+\lambda+\mu>0,\;\Re z+\eta>0$
\begin{align*}
\hat{p}_1(0,t)&=\frac{2\lambda}{\eta-\lambda-\mu}e^{-(\lambda+\mu)t}-\frac{2\lambda}{\eta-\lambda-\mu}e^{-\eta t}\\
&\quad+\frac{\lambda\eta}{\eta-\lambda-\mu}e^{-(\lambda+\mu)t}\int^t_0\hat{p}_1(0,\tau)e^{(\lambda+\mu)\tau}\mathrm{d}\tau\\
&\quad-\frac{\lambda\eta}{\eta-\lambda-\mu}e^{-\eta t}\int^t_0\hat{p}_1(0,\tau)e^{\eta\tau}\mathrm{d}\tau\\
&\Longrightarrow\\
\int^{\infty}_0e^{-tz}\hat{p}_1(0,t)\mathrm{d}t&=\frac{2\lambda}{\eta-\lambda-\mu}\int^{\infty}_0e^{-tz}e^{-(\lambda+\mu)t}\mathrm{d}t
-\frac{2\lambda}{\eta-\lambda-\mu}\int^{\infty}_0e^{-tz}e^{-\eta t}\mathrm{d}t\\
&\quad+\frac{\lambda\eta}{\eta-\lambda-\mu}\int^{\infty}_0e^{-tz}e^{-(\lambda+\mu)t}\int^t_0\hat{p}_1(0,\tau)e^{(\lambda+\mu)\tau}\mathrm{d}\tau\mathrm{d}t\\
&\quad-\frac{\lambda\eta}{\eta-\lambda-\mu}\int^{\infty}_0e^{-tz}e^{-\eta t}\int^t_0\hat{p}_1(0,\tau)e^{\eta\tau}\mathrm{d}\tau\mathrm{d}t\\
&=\frac{2\lambda}{\eta-\lambda-\mu}\frac{-1}{\lambda+\mu+z}e^{-(\lambda+\mu+z)t}\Big\vert^{t=\infty}_{t=0}\\
&\quad-\frac{2\lambda}{\eta-\lambda-\mu}\frac{-1}{\eta+z}e^{-(\eta+z)t}\Big\vert^{t=\infty}_{t=0}\\
&\quad+\frac{\lambda\eta}{\eta-\lambda-\mu}\int^{\infty}_0\hat{p}_1(0,\tau)e^{(\lambda+\mu)\tau}\int^{\infty}_{\tau}e^{-(\lambda+\mu+z)t}\mathrm{d}t\mathrm{d}\tau\\
&\quad-\frac{\lambda\eta}{\eta-\lambda-\mu}\int^{\infty}_0\hat{p}_1(0,\tau)e^{\eta\tau}\int^{\infty}_{\tau}e^{-(z+\eta)t}\mathrm{d}t\mathrm{d}\tau\\
&=\frac{2\lambda}{\eta-\lambda-\mu}\frac{1}{\lambda+\mu+z}-\frac{2\lambda}{\eta-\lambda-\mu}\frac{1}{\eta+z}\\
&\quad+\frac{\lambda\eta}{\eta-\lambda-\mu}\int^{\infty}_0\frac{-1}{\lambda+\mu+z}\hat{p}_1(0,\tau)e^{(\lambda+\mu)\tau}
e^{-(\lambda+\mu+z)t}\Big\vert^{t=\infty}_{t=\tau}\mathrm{d}\tau\\
&\quad-\frac{\lambda\eta}{\eta-\lambda-\mu}\int^{\infty}_0\hat{p}_1(0,\tau)e^{\eta\tau}\frac{-e^{-(\eta+z)t}}{\eta+z}\bigg\vert^{t=\infty}_{t=\tau}\mathrm{d}\tau\\
&=\frac{2\lambda}{\eta-\lambda-\mu}\frac{1}{\lambda+\mu+z}-\frac{2\lambda}{\eta-\lambda-\mu}\frac{1}{\eta+z}\\
&\quad+\frac{\lambda\eta}{(\eta-\lambda-\mu)(\lambda+\mu+z)}\int^{\infty}_0\hat{p}_1(0,\tau)e^{-z\tau}\mathrm{d}\tau\\
&\quad-\frac{\lambda\eta}{(\eta-\lambda-\mu)(\eta+z)}\int^{\infty}_0\hat{p}_1(0,\tau)e^{-z\tau}\mathrm{d}\tau\\
&\Longrightarrow\\
&\left[1-\frac{\lambda\eta}{(\eta+z)(\lambda+\mu+z)}\right]\int^{\infty}_0\hat{p}_1(0,t)e^{-tz}\mathrm{d}t\\
&=\frac{2\lambda}{(\eta-\lambda-\mu)(\lambda+\mu+z)}-\frac{2\lambda}{(\eta-\lambda-\mu)(\eta+z)}\\
&\Longrightarrow\\
&\frac{(\eta+z)(\lambda+\mu+z)-\lambda\eta}{(\eta+z)(\lambda+\mu+z)}\int^{\infty}_0\hat{p}_1(0,t)e^{-tz}\mathrm{d}t\\
&=\frac{2\lambda}{(\eta+z)(\lambda+\mu+z)}\\
&\Longrightarrow\\
\int^{\infty}_0\hat{p}_1(0,t)e^{-tz}\mathrm{d}t&=\frac{2\lambda}{(\eta+z)(\lambda+\mu+z)-\lambda\eta}\\
&=\frac{2\lambda}{z^2+(\lambda+\mu+\eta)z+\mu\eta}\\
&=\frac{h_1}{z-z_1}+\frac{h_2}{z-z_2},\tag {2.28}\label{2.28}
\end{align*}
here
\begin{align}
h_1&=\frac{2\lambda}{z_1-z_2},\;h_2=-\frac{2\lambda}{z_1-z_2}.\tag {2.29}\label{2.29}
\end{align}
(\ref{2.28}) gives
\begin{align}
\hat{p}_1(0,t)&=h_1e^{z_1t}+h_2e^{z_2t}.\tag {2.30}\label{2.30}
\end{align}
By inserting (\ref{2.30}) into (\ref{2.27}) we obtain
\begin{align*}
\hat{p}_1(s,t)&=\frac{2\lambda}{\eta+s-\lambda-\mu}e^{-(\lambda+\mu)t}-\frac{2\lambda}{\eta+s-\lambda-\mu}e^{-(\eta+s)t}\\
&\quad+\frac{\lambda\eta}{\eta+s-\lambda-\mu}e^{-(\lambda+\mu)t}\int^t_0\left[h_1e^{z_1\tau}+h_2e^{z_2\tau}\right]e^{(\lambda+\mu)\tau}\mathrm{d}\tau\\
&\quad-\frac{\lambda\eta}{\eta+s-\lambda-\mu}e^{-(\eta+s)t}\int^t_0\left[h_1e^{z_1\tau}+h_2e^{z_2\tau}\right]e^{(\eta+s)\tau}\mathrm{d}\tau\\
&=\frac{2\lambda}{\eta+s-\lambda-\mu}e^{-(\lambda+\mu)t}-\frac{2\lambda}{\eta+s-\lambda-\mu}e^{-(\eta+s)t}\\
&\quad+\frac{\lambda\eta h_1}{\eta+s-\lambda-\mu}e^{-(\lambda+\mu)t}\int^t_0e^{(\lambda+\mu+z_1)\tau}\mathrm{d}\tau\\
&\quad+\frac{\lambda\eta h_2}{\eta+s-\lambda-\mu}e^{-(\lambda+\mu)t}\int^t_0e^{(\lambda+\mu+z_2)\tau}\mathrm{d}\tau\\
&\quad-\frac{\lambda\eta h_1}{\eta+s-\lambda-\mu}e^{-(\eta+s)t}\int^t_0e^{(\eta+z_1+s)\tau}\mathrm{d}\tau\\
&\quad-\frac{\lambda\eta h_2}{\eta+s-\lambda-\mu}e^{-(\eta+s)t}\int^t_0e^{(\eta+z_2+s)\tau}\mathrm{d}\tau\\
&=\frac{2\lambda}{\eta+s-\lambda-\mu}e^{-(\lambda+\mu)t}-\frac{2\lambda}{\eta+s-\lambda-\mu}e^{-(\eta+s)t}\\
&\quad+\frac{\lambda\eta h_1}{\eta+s-\lambda-\mu}e^{-(\lambda+\mu)t}\frac{1}{\lambda+\mu+z_1}\left[e^{(\lambda+\mu+z_1)t}-1\right]\\
&\quad+\frac{\lambda\eta h_2}{\eta+s-\lambda-\mu}e^{-(\lambda+\mu)t}\frac{1}{\lambda+\mu+z_2}\left[e^{(\lambda+\mu+z_2)t}-1\right]\\
&\quad-\frac{\lambda\eta h_1}{\eta+s-\lambda-\mu}e^{-(\eta+s)t}\frac{1}{\eta+z_1+s}\left[e^{(\eta+z_1+s)t}-1\right]\\
&\quad-\frac{\lambda\eta h_2}{\eta+s-\lambda-\mu}e^{-(\eta+s)t}\frac{1}{\eta+z_2+s}\left[e^{(\eta+z_2+s)t}-1\right]\\
&=\frac{2\lambda}{\eta-\lambda-\mu+s}e^{-(\lambda+\mu)t}-\frac{2\lambda}{\eta-\lambda-\mu+s}e^{-st}e^{-\eta t}\\
&\quad+\frac{\lambda\eta h_1}{\lambda+\mu+z_1}e^{z_1t}\frac{1}{\eta-\lambda-\mu+s}\\
&\quad-\frac{\lambda\eta h_1}{\lambda+\mu+z_1}e^{-(\lambda+\mu)t}\frac{1}{\eta-\lambda-\mu+s}\\
&\quad+\frac{\lambda\eta h_2}{\lambda+\mu+z_2}e^{z_2t}\frac{1}{\eta-\lambda-\mu+s}\\
&\quad-\frac{\lambda\eta h_2}{\lambda+\mu+z_2}e^{-(\lambda+\mu)t}\frac{1}{\eta-\lambda-\mu+s}\\
&\quad-\frac{\lambda\eta h_1}{(\eta-\lambda-\mu+s)(\eta+z_1+s)}e^{z_1t}\\
&\quad+\frac{\lambda\eta h_1}{(\eta-\lambda-\mu+s)(\eta+z_1+s)}e^{-(\eta+s)t}\\
&\quad-\frac{\lambda\eta h_2}{(\eta-\lambda-\mu+s)(\eta+z_2+s)}e^{z_2t}\\
&\quad+\frac{\lambda\eta h_2}{(\eta-\lambda-\mu+s)(\eta+z_2+s)}e^{-(\eta+s)t}\\
&=\frac{2\lambda}{\eta-\lambda-\mu+s}e^{-(\lambda+\mu)t}-\frac{2\lambda}{\eta-\lambda-\mu+s}e^{-st}e^{-\eta t}\\
&\quad+\frac{\lambda\eta h_1}{\lambda+\mu+z_1}e^{z_1t}\frac{1}{\eta-\lambda-\mu+s}\\
&\quad-\frac{\lambda\eta h_1}{\lambda+\mu+z_1}e^{-(\lambda+\mu)t}\frac{1}{\eta-\lambda-\mu+s}\\
&\quad+\frac{\lambda\eta h_2}{\lambda+\mu+z_2}e^{z_2t}\frac{1}{\eta-\lambda-\mu+s}\\
&\quad-\frac{\lambda\eta h_2}{\lambda+\mu+z_2}e^{-(\lambda+\mu)t}\frac{1}{\eta-\lambda-\mu+s}\\
&\quad-\frac{\lambda\eta h_1}{\lambda+\mu+z_1}e^{z_1t}\frac{1}{\eta-\lambda-\mu+s}\\
&\quad+\frac{\lambda\eta h_1}{\lambda+\mu+z_1}e^{z_1t}\frac{1}{\eta+z_1+s}\\
&\quad+\frac{\lambda\eta h_1}{\lambda+\mu+z_1}e^{-\eta t}\frac{1}{\eta-\lambda-\mu+s}e^{-st}\\
&\quad-\frac{\lambda\eta h_1}{\lambda+\mu+z_1}e^{-\eta t}\frac{1}{\eta+z_1+s}e^{-st}\\
&\quad-\frac{\lambda\eta h_2}{\lambda+\mu+z_2}e^{z_2t}\frac{1}{\eta-\lambda-\mu+s}\\
&\quad+\frac{\lambda\eta h_2}{\lambda+\mu+z_2}e^{z_2t}\frac{1}{\eta+z_2+s}\\
&\quad+\frac{\lambda\eta h_2}{\lambda+\mu+z_2}e^{-\eta t}\frac{1}{\eta-\lambda-\mu+s}e^{-st}\\
&\quad-\frac{\lambda\eta h_2}{\lambda+\mu+z_2}e^{-\eta t}\frac{1}{\eta+z_2+s}e^{-st}.\tag {2.31}\label{2.31}
\end{align*}
By combining (\ref{2.23}), (\ref{2.24}) and (\ref{2.25}) with (\ref{2.31}) we determine the Laplace inverse transform
\begin{align*}
p_1(x,t)&=2\lambda e^{(\lambda+\mu-\eta)x}e^{-(\lambda+\mu)t}-2\lambda e^{-\eta t}\int^x_0e^{(\lambda+\mu-\eta)(x-z)}\delta(z-t)\mathrm{d}z\\
&\quad+\frac{\lambda\eta h_1}{\lambda+\mu+z_1}e^{z_1t}e^{(\lambda+\mu-\eta)x}\\
&\quad-\frac{\lambda\eta h_1}{\lambda+\mu+z_1}e^{-(\lambda+\mu)t}e^{(\lambda+\mu-\eta)x}\\
&\quad+\frac{\lambda\eta h_2}{\lambda+\mu+z_2}e^{z_2t}e^{(\lambda+\mu-\eta)x}\\
&\quad-\frac{\lambda\eta h_2}{\lambda+\mu+z_2}e^{-(\lambda+\mu)t}e^{(\lambda+\mu-\eta)x}\\
&\quad-\frac{\lambda\eta h_1}{\lambda+\mu+z_1}e^{z_1t}e^{(\lambda+\mu-\eta)x}\\
&\quad+\frac{\lambda\eta h_1}{\lambda+\mu+z_1}e^{z_1t}e^{-(\eta+z_1)x}\\
&\quad+\frac{\lambda\eta h_1}{\lambda+\mu+z_1}e^{-\eta t}\int^x_0e^{(\lambda+\mu-\eta)(x-z)}\delta(z-t)\mathrm{d}z\\
&\quad-\frac{\lambda\eta h_1}{\lambda+\mu+z_1}e^{-\eta t}\int^x_0e^{-(\eta+z_1)(x-z)}\delta(z-t)\mathrm{d}z\\
&\quad-\frac{\lambda\eta h_2}{\lambda+\mu+z_2}e^{z_2t}e^{(\lambda+\mu-\eta)x}\\
&\quad-\frac{\lambda\eta h_2}{\lambda+\mu+z_2}e^{z_2t}e^{-(\eta+z_2)x}\\
&\quad+\frac{\lambda\eta h_2}{\lambda+\mu+z_2}e^{-\eta t}\int^x_0e^{(\lambda+\mu-\eta)(x-z)}\delta(z-t)\mathrm{d}z\\
&\quad-\frac{\lambda\eta h_2}{\lambda+\mu+z_2}e^{-\eta t}\int^x_0e^{-(\eta+z_2)(x-z)}\delta(z-t)\mathrm{d}z.\tag {2.32}\label{2.32}
\end{align*}
Since $\;\lambda>0,\;\mu>0,\;\eta>0,\;\Re z_1<0,\;\Re z_2<0,$ it is easy to see from (\ref{2.32}) and (\ref{2.26})
\begin{align}
&\lim_{t\to\infty}p_0(x,t)=0,\quad \lim_{t\to\infty}p_1(x,t)=0.\tag{2.33}\label{2.33}
\end{align}
This together with (\ref{2.26}) and (\ref{2.32}) are just the result of this theorem.
\end{proof}

The result (\ref{2.33}) is quite different from the asymptotic behavior of the time-dependent solutions of many reliability models \cite{Gupur2011,Gupur2014,Gupur2016,Gupur2020,Gupur2022,Gupur2005,GupurWong,HajiRadl,HuXuYuZhu}.

\section{General Case}\label{sec2}

\begin{lemma}
If
\begin{align*}
&\int^{\infty}_0\frac{\partial p_k(x,t)}{\partial x}e^{-st}\mathrm{d}t=\frac{\mathrm{d}}{\mathrm{d}x}\int^{\infty}_0p_k(x,t)e^{-st}\mathrm{d}t,\; k=0,1,
\end{align*}
then the time-dependent solution of the equation system (\ref{1.1})$\sim$(\ref{1.5}) satisfies
\begin{align*}
\hat{p}_1(x,s)&=e^{-sx-\int^x_0\eta(\tau)\mathrm{d}\tau}\bigg\{\int^{\infty}_0\lambda(x)e^{-sx-\int^x_0(\lambda(\tau)+\mu(\tau))\mathrm{d}\tau}\mathrm{d}x\\
&\quad+\int^{\infty}_0\lambda(x)e^{-sx-\int^x_0(\lambda(\tau)+\mu(\tau))\mathrm{d}\tau}\\
&\quad\times\int^x_0\delta(y)e^{sy+\int^y_0(\lambda(\tau)+\mu(\tau))\mathrm{d}\tau}\mathrm{d}y\mathrm{d}x\bigg\}\\
&\quad\times\bigg[1-\int^{\infty}_0\lambda(x)e^{-sx-\int^x_0(\lambda(\tau)+\mu(\tau))\mathrm{d}\tau}\mathrm{d}x\\
&\quad\times\int^{\infty}_0\eta(x)e^{-sx-\int^x_0\eta(\tau)\mathrm{d}\tau}\mathrm{d}x\bigg]^{-1},\\
\hat{p}_0(x,s)&=e^{-sx-\int^x_0(\lambda(\tau)+\mu(\tau))\mathrm{d}\tau}\\
&\quad+e^{-sx-\int^x_0(\lambda(\tau)+\mu(\tau))\mathrm{d}\tau}\int^{\infty}_0\eta(x)e^{-sx-\int^x_0\eta(\tau)\mathrm{d}\tau}\mathrm{d}x\\
&\quad\times\bigg\{\int^{\infty}_0\lambda(x)e^{-sx-\int^x_0(\lambda(\tau)+\mu(\tau))\mathrm{d}\tau}\mathrm{d}x\\
&\quad+\int^{\infty}_0\lambda(x)e^{-sx-\int^x_0(\lambda(\tau)+\mu(\tau))\mathrm{d}\tau}\\
&\quad\times\int^x_0\delta(y)e^{sy+\int^y_0(\lambda(\tau)+\mu(\tau))\mathrm{d}\tau}\mathrm{d}y\mathrm{d}x\bigg\}\\
&\quad\times\bigg[1-\int^{\infty}_0\lambda(x)e^{-sx-\int^x_0(\lambda(\tau)+\mu(\tau))\mathrm{d}\tau}\mathrm{d}x\\
&\quad\times\int^{\infty}_0\eta(x)e^{-sx-\int^x_0\eta(\tau)\mathrm{d}\tau}\mathrm{d}x\bigg]^{-1}.
\end{align*}
\end{lemma}

\begin{proof}
We consider the Laplace transform of the time-dependent solution of the equation system (\ref{1.1})$\sim$(\ref{1.5}):
\begin{align}
&\frac{\partial p_0(x,t)}{\partial t}+\frac{\partial p_0(x,t)}{\partial x}=-(\lambda(x)+\mu(x))p_0(x,t),\tag {3.1}\label{3.1}\\
&\frac{\partial p_1(x,t)}{\partial t}+\frac{\partial p_1(x,t)}{\partial x}=-\eta(x)p_1(x,t),\tag {3.2}\label{3.2}\\
&p_0(0,t)=\delta(t)+\int^{\infty}_0p_1(x,t)\eta(x)\mathrm{d}x,\tag {3.3}\label{3.3}\\
&p_1(0,t)=\int^{\infty}_0p_0(x,t)\lambda(x)\mathrm{d}x,\tag {3.4}\label{3.4}\\
&p_0(x,0)=\delta(x),\quad p_1(x,0)=0.\tag {3.5}\label{3.5}
\end{align}
By applying the Laplace transform with respect to $t$ to (\ref{3.1}) and (\ref{3.2}), and using (\ref{2.6}), (\ref{2.7}) and (\ref{3.5}) we derive
\begin{align*}
&s\hat{p}_0(x,s)-p_0(x,0)+\frac{\partial \hat{p}_0(x,s)}{\partial x}=-(\lambda(x)+\mu(x))\hat{p}_0(x,s)\\
&\Longrightarrow\\
&\frac{\partial\hat{p}_0(x,s)}{\partial x}=-(\lambda(x)+\mu(x)+s)\hat{p}_0(x,s)+p_0(x,0)\\
&\Longrightarrow\\
\hat{p}_0(x,s)&=\hat{p}_0(0,s)e^{-\int^x_0(\lambda(\tau)+\mu(\tau))\mathrm{d}\tau-sx}\\
&\quad+e^{-\int^x_0(\lambda(\tau)+\mu(\tau))\mathrm{d}\tau-sx}\int^x_0p_0(y,0)e^{\int^y_0(\lambda(\tau)+\mu(\tau))\mathrm{d}\tau+sy}\mathrm{d}y\\
&=\hat{p}_0(0,s)e^{-\int^x_0(\lambda(\tau)+\mu(\tau))\mathrm{d}\tau-sx}\\
&\quad+e^{-\int^x_0(\lambda(\tau)+\mu(\tau))\mathrm{d}\tau-sx}\int^x_0\delta(y)e^{\int^y_0(\lambda(\tau)+\mu(\tau))\mathrm{d}\tau+sy}\mathrm{d}y.\tag {3.6}\label{3.6}
\end{align*}
\begin{align*}
&s\hat{p}_1(x,s)-p_1(x,0)+\frac{\partial\hat{p}_1(x,s)}{\partial x}=-\eta(x)\hat{p}_1(x,s)\\
&\Longrightarrow\\
&\frac{\partial\hat{p}_1(x,s)}{\partial x}=-(\eta(x)+s)\hat{p}_1(x,s)+p_1(x,0)\\
&\Longrightarrow\\
\hat{p}_1(x,s)&=\hat{p}_1(0,s)e^{-sx-\int^x_0\eta(\tau)\mathrm{d}\tau}\\
&\quad+e^{-sx-\int^x_0\eta(\tau)d\tau}\int^x_0p_1(y,0)e^{sy+\int^y_0\eta(\tau)\mathrm{d}\tau}\mathrm{d}y\\
&=\hat{p}_1(0,s)e^{-sx-\int^x_0\eta(\tau)\mathrm{d}\tau}.\tag {3.7}\label{3.7}
\end{align*}
(\ref{3.3}) and the Fubini theorem give
\begin{align*}
\hat{p}_0(0,s)&=\int^{\infty}_0p_0(0,t)e^{-st}\mathrm{d}t\\
&=\int^{\infty}_0e^{-st}\left[\delta(t)+\int^{\infty}_0p_1(x,t)\eta(x)\mathrm{d}x\right]\mathrm{d}t\\
&=\int^{\infty}_0e^{-st}\delta(t)\mathrm{d}t+\int^{\infty}_0e^{-st}\int^{\infty}_0p_1(x,t)\eta(x)\mathrm{d}x\mathrm{d}t\\
&=1+\int^{\infty}_0\int^{\infty}_0p_1(x,t)\eta(x)e^{-st}\mathrm{d}t\mathrm{d}x\\
&=1+\int^{\infty}_0\eta(x)\hat{p}_1(x,s)\mathrm{d}x.\tag {3.8}\label{3.8}
\end{align*}
(\ref{3.4}) and the Fubini theorem imply
\begin{align*}
\hat{p}_1(0,s)&=\int^{\infty}_0p_1(0,t)e^{-st}\mathrm{d}t\\
&=\int^{\infty}_0e^{-st}\int^{\infty}_0p_0(x,t)\lambda(x)\mathrm{d}x\mathrm{d}t\\
&=\int^{\infty}_0\int^{\infty}_0p_0(x,t)\lambda(x)e^{-st}\mathrm{d}t\mathrm{d}x\\
&=\int^{\infty}_0\hat{p}_0(x,s)\lambda(x)\mathrm{d}x.\tag {3.9}\label{3.9}
\end{align*}
By inserting  (\ref{3.8}) and (\ref{3.9}) into (\ref{3.6}) and (\ref{3.7}) we have
\begin{align*}
\hat{p}_0(x,s)&=\left[1+\int^{\infty}_0\eta(x)\hat{p}_1(x,s)\mathrm{d}x\right]e^{-sx-\int^x_0(\lambda(\tau)+\mu(\tau))\mathrm{d}\tau}\\
&\quad+e^{-sx-\int^x_0(\lambda(\tau)+\mu(\tau))\mathrm{d}\tau}\int^x_0\delta(y)e^{sy+\int^y_0(\lambda(\tau)+\mu(\tau))\mathrm{d}\tau}\mathrm{d}y\\
&=e^{-sx-\int^x_0(\lambda(\tau)+\mu(\tau))\mathrm{d}\tau}\\
&\quad+e^{-sx-\int^x_0(\lambda(\tau)+\mu(\tau))\mathrm{d}\tau}\int^{\infty}_0\eta(x)\hat{p}_1(x,s)\mathrm{d}x\\
&\quad+e^{-sx-\int^x_0(\lambda(\tau)+\mu(\tau))\mathrm{d}\tau}\int^x_0\delta(y)e^{sy+\int^y_0(\lambda(\tau)+\mu(\tau))\mathrm{d}\tau}\mathrm{d}y.\tag {3.10}\label{3.10}\\
\hat{p}_1(x,s)&=e^{-sx-\int^x_0\eta(\tau)\mathrm{d}\tau}\int^{\infty}_0\hat{p}_0(x,s)\lambda(x)\mathrm{d}x.\tag {3.11}\label{3.11}
\end{align*}
By substituting (\ref{3.11}) into (\ref{3.10}) we determine
\begin{align*}
\hat{p}_0(x,s)&=e^{-sx-\int^x_0(\lambda(\tau)+\mu(\tau))\mathrm{d}\tau}\\
&\quad+e^{-sx-\int^x_0(\lambda(\tau)+\mu(\tau))\mathrm{d}\tau}\\
&\quad\times\int^{\infty}_0\eta(x)e^{-sx-\int^x_0\eta(\tau)\mathrm{d}\tau}\mathrm{d}x\int^{\infty}_0\hat{p}_0(x,s)\lambda(x)\mathrm{d}x\\
&\quad+e^{-sx-\int^x_0(\lambda(\tau)+\mu(\tau))\mathrm{d}\tau}\\
&\quad\times\int^x_0\delta(y)e^{sy+\int^y_0(\lambda(\tau)+\mu(\tau))\mathrm{d}\tau}\mathrm{d}y\\
&\Longrightarrow\\
\int^{\infty}_0\hat{p}_0(x,s)\lambda(x)\mathrm{d}x
&=\int^{\infty}_0\lambda(x)e^{-sx-\int^x_0(\lambda(\tau)+\mu(\tau))\mathrm{d}\tau}\mathrm{d}x\\
&\quad+\int^{\infty}_0\lambda(x)e^{-sx-\int^x_0(\lambda(\tau)+\mu(\tau))\mathrm{d}\tau}\mathrm{d}x\\
&\quad\times\int^{\infty}_0\eta(x)e^{-sx-\int^x_0\eta(\tau)\mathrm{d}\tau}\mathrm{d}x\int^{\infty}_0\hat{p}_0(x,s)\lambda(x)\mathrm{d}x\\
&\quad+\int^{\infty}_0\lambda(x)e^{-sx-\int^x_0(\lambda(\tau)+\mu(\tau))\mathrm{d}\tau}\\
&\quad\times\int^x_0\delta(y)e^{sy+\int^y_0(\lambda(\tau)+\mu(\tau))\mathrm{d}\tau}\mathrm{d}y\mathrm{d}x\\
&\Longrightarrow\\
&\bigg[1-\int^{\infty}_0\lambda(x)e^{-sx-\int^x_0(\lambda(\tau)+\mu(\tau))\mathrm{d}\tau}\mathrm{d}x\\
&\quad\times\int^{\infty}_0\eta(x)e^{-sx-\int^x_0\eta(\tau)d\tau}\mathrm{d}x\bigg]\int^{\infty}_0\hat{p}_0(x,s)\lambda(x)\mathrm{d}x\\
&=\int^{\infty}_0\lambda(x)e^{-sx-\int^x_0(\lambda(\tau)+\mu(\tau))\mathrm{d}\tau}\mathrm{d}x\\
&\quad+\int^{\infty}_0\lambda(x)e^{-sx-\int^x_0(\lambda(\tau)+\mu(\tau))\mathrm{d}\tau}\\
&\quad\times\int^x_0\delta(y)e^{sy+\int^y_0(\lambda(\tau)+\mu(\tau))\mathrm{d}\tau}\mathrm{d}y\mathrm{d}x\\
&\Longrightarrow\\
\int^{\infty}_0\hat{p}_0(x,s)\lambda(x)\mathrm{d}x
&=\bigg\{\int^{\infty}_0\lambda(x)e^{-sx-\int^x_0(\lambda(\tau)+\mu(\tau))\mathrm{d}\tau}\mathrm{d}x\\
&\quad+\int^{\infty}_0\lambda(x)e^{-sx-\int^x_0(\lambda(\tau)+\mu(\tau))\mathrm{d}\tau}\\
&\quad\times\int^x_0\delta(y)e^{sy+\int^y_0(\lambda(\tau)+\mu(\tau))\mathrm{d}\tau}\mathrm{d}y\mathrm{d}x\bigg\}\\
&\quad\times\Big[1-\int^{\infty}_0\lambda(x)e^{-sx-\int^x_0(\lambda(\tau)+\mu(\tau))\mathrm{d}\tau}\mathrm{d}x\\
&\quad\times\int^{\infty}_0\eta(x)e^{-sx-\int^x_0\eta(\tau)\mathrm{d}\tau}\mathrm{d}x\Big]^{-1}.\tag {3.12}\label{3.12}
\end{align*}
By combining (\ref{3.12}) with (\ref{3.11}) we calculate
\begin{align*}
\hat{p}_1(x,s)&=e^{-sx-\int^x_0\eta(\tau)d\tau}\int^{\infty}_0\hat{p}_0(x,s)\lambda(x)\mathrm{d}x\\
&=e^{-sx-\int^x_0\eta(\tau)\mathrm{d}\tau}\bigg\{\int^{\infty}_0\lambda(x)e^{-sx-\int^x_0(\lambda(\tau)+\mu(\tau))\mathrm{d}\tau}\mathrm{d}x\\
&\quad+\int^{\infty}_0\lambda(x)e^{-sx-\int^x_0(\lambda(\tau)+\mu(\tau))\mathrm{d}\tau}\\
&\quad\times\int^x_0\delta(y)e^{sy+\int^y_0(\lambda(\tau)+\mu(\tau))\mathrm{d}\tau}\mathrm{d}y\mathrm{d}x\bigg\}\\
&\quad\times\Big[1-\int^{\infty}_0\lambda(x)e^{-sx-\int^x_0(\lambda(\tau)+\mu(\tau))\mathrm{d}\tau}\mathrm{d}x\\
&\quad\times\int^{\infty}_0\eta(x)e^{-sx-\int^x_0\eta(\tau)\mathrm{d}\tau}\mathrm{d}x\Big]^{-1}.\tag {3.13}\label{3.13}
\end{align*}
By inserting (\ref{3.13}) into (\ref{3.10}) we have
\begin{align*}
\hat{p}_0(x,s)&=e^{-sx-\int^x_0(\lambda(\tau)+\mu(\tau))\mathrm{d}\tau}\\
&\quad+e^{-sx-\int^x_0(\lambda(\tau)+\mu(\tau))\mathrm{d}\tau}\int^{\infty}_0\eta(x)e^{-sx-\int^x_0\eta(\tau)\mathrm{d}\tau}\mathrm{d}x\\
&\quad\times\bigg\{\int^{\infty}_0\lambda(x)e^{-sx-\int^x_0(\lambda(\tau)+\mu(\tau))\mathrm{d}\tau}\mathrm{d}x\\
&\quad+\int^{\infty}_0\lambda(x)e^{-sx-\int^x_0(\lambda(\tau)+\mu(\tau))\mathrm{d}\tau}\\
&\quad\times\int^x_0\delta(y)e^{sy+\int^y_0(\lambda(\tau)+\mu(\tau))\mathrm{d}\tau}\mathrm{d}y\mathrm{d}x\bigg\}\\
&\quad\times\Big[1-\int^{\infty}_0\lambda(x)e^{-sx-\int^x_0(\lambda(\tau)+\mu(\tau))\mathrm{d}\tau}\mathrm{d}x\\
&\quad\times\int^{\infty}_0\eta(x)e^{-sx-\int^x_0\eta(\tau)\mathrm{d}\tau}\mathrm{d}x\Big]^{-1}.\tag {3.14}\label{3.14}
\end{align*}
(\ref{3.13}) and (\ref{3.14}) are the result of the lemma.

\end{proof}

\begin{remark}
If we can determine the Laplace inverse transform of (\ref{3.13}) and (\ref{3.14}), then we can get the time-dependent solution of the equation system (\ref{1.1})$\sim$(\ref{1.5}). But, until today we have not discovered
an efficient method to resolve this problem.
\end{remark}

If we regard the Dirac-delta function $\;\delta\in L^1[0,\infty)$ and select
$$
X=\left\{p\in L^1[0,\infty)\times L^1[0,\infty)\;\Big\vert\; \|p\|=\sum^1_{i=0}\|p_i\|_{L^1[0,\infty)}\right\}
$$
as a state space and define an operator as
$$
(A(t)p)(x)=\left(
\begin{matrix}
-\frac{\mathrm{d}}{\mathrm{d}x}-(\lambda(x)+\mu(x)) &0\\
0 &-\frac{\mathrm{d}}{\mathrm{d}x}-\eta(x)
\end{matrix}
\right)
\left(
\begin{matrix}
p_0(x)\\
p_1(x)
\end{matrix}
\right),
$$
$$
D(A(t))=\left\{p\in X\;\Big\vert
\begin{array}{ll}
p_i(x)\;(i=0,1)\;\text{are absolutely continuous}\\
\text{and satisfy}\;p(0)=(\delta(\cdot),0)^T+\int^{\infty}_0\Gamma p(x)dx
\end{array}
\right\},
$$
where
$$
\Gamma=\left(
\begin{matrix}
0 &\eta(x)\\
\lambda(x) &0
\end{matrix}
\right),
$$
then we convert the equation system (\ref{1.1})$\sim$(\ref{1.5}) as an abstract Cauchy problem in the Banach space $\;X:$
\begin{align}
\left\{
\begin{array}{ll}
\frac{\mathrm{d}p(t)}{\mathrm{d}t}=A(t)p(t),\;\forall t\in (0,\infty)\\
p(0)=(\delta(x),0)
\end{array}
\right.\tag{3.15}\label{3.15}
\end{align}
The system (\ref{3.15}) is quite different from the case in \cite{Gupur2011,Gupur2020} where many reliability models have been resolved by using $C_0-$semigroup theory \cite{Arendt2001,Pazy,SongYu,Webb}. The main difficult point of (\ref{3.15}) is that $A(t)$ is an evolution family of $t.$

\section{Open Problems}

One may find the following reliability models in Chapter 5 of \cite{Gupur2011} where the models were left as open problems.

In 1976,   Linton \cite{Linton1976} described a two-unit parallel redundant system by the following system of equations:
\begin{align*}
&\frac{\partial q_{0,1}(x,t)}{\partial t}+\frac{\partial q_{0,1}(x,t)}{\partial x}=\mu q_{2,m}(x,t),\\
&\frac{\partial q_{0,i}(x,t)}{\partial t}+\frac{\partial q_{0,i}(x,t)}{\partial x}=-\lambda q_{0,i}(x,t)
+\lambda q_{0,i-1}(x,t),\quad 2\le i\le k,\\
&\frac{\partial q_{1,1}(x,t)}{\partial t}+\frac{\partial q_{1,1}(x,t)}{\partial x}=0,\\
&\frac{\partial q_{1,2}(x,t)}{\partial t}+\frac{\partial q_{1,2}(x,t)}{\partial x}=-\lambda q_{1,2}(x,t)+\lambda q_{1,1}(x,t),\\
&\frac{\partial q_{1,i}(x,t)}{\partial t}+\frac{\partial q_{1,i}(x,t)}{\partial x}=-\lambda q_{1,i}(x,t)+\lambda q_{1,i-1}(x,t),\quad 3\le i\le k,\\
&\frac{\partial q_{2,1}(x,t)}{\partial t}+\frac{\partial q_{2,1}(x,t)}{\partial x}=-\mu q_{2,1}(x,t)+\lambda q_{0,k}(x,t),\\
&\frac{\partial q_{2,j}(x,t)}{\partial t}+\frac{\partial q_{2,j}(x,t)}{\partial x}=-\mu q_{2,j}(x,t)+\mu q_{2,j-1}(x,t),\quad 2\le j\le m,\\
&q_{0,1}(0,t)=\int^{\infty}_0q_{1,1}(x,t)g(x)\mathrm{d}x+\delta(t),\\
&q_{0,i}(0,t)=\int^{\infty}_0q_{1,i}(x,t)g(x)\mathrm{d}x,\quad 2\le i\le k,\\
&q_{1,i}(0,t)=\int^{\infty}_0q_{0,i}(x,t)f(x)\mathrm{d}x,\quad 1\le i\le k,\\
&q_{2,j}(0,t)=0,\quad 1\le j\le m,\\
&q_{0,1}(x,0)=\delta(x),\quad q_{0,i}(x,0)=0,\quad 2\le i\le k\\
&q_{1,i}(x,0)=0,\; 1\le i\le k;\; q_{2,j}(x,0)=0,\; 1\le j\le m.
\end{align*}
Where $(x,t)\in [0,\infty)\times[0,\infty);$ $p_{0,i}(x,t)\; (1\le i\le k)$ is the probability that at moment $t,$ no units are down, unit 2 is in the $i$th stage of
failure, and unit 1 has been operating for an amount of time $x;$ $p_{1,i}(x,t)\; (1\le i\le k)$ is the probability that at moment $t,$ unit 1 is down, unit 2 is
in the $i$th stage of failure and an amount of $x,$ has already been spent on the current repair of unit 1;
$p_{2,j}(x,t)\; (1\le j\le m)$ are the probability that at moment $t,$ unit 2 is down, unit 2 is in the $j$th phase of repair, and unit 1
has been operating for an amount of time $x;$ $\lambda>0$ and $\mu>0$ are parameters of unit 1 and unit 2 respectively; $g(x)$ is the repair rate of unit 1;
$f(x)$ is the failure rate of unit 2, $\delta(\cdot)$ is the Dirac-delta function.

When $g(x)$ and $f(x)$ are constants, by using our methods of theorem \ref{geni} one may determine the explicit form of the time-dependent solution of the above model. When  $g(x)$ and $f(x)$ are non-constants, well-posedness of the above model and asymptotic behavior of its time-dependent solution have not been obtained until today.

In 1993, Li et al. \cite{LiShi} established the following model:
\begin{align*}
&\frac{\partial p_{0,1}(x,t)}{\partial t}+\frac{\partial p_{0,1}(x,t)}{\partial x}=-(a_1+a_2(x))p_{0,1}(x,t)\\
&\hspace{4cm}+\int^x_0p_2(x,u,t)b_1(u)\mathrm{d}u,\\
&\frac{\partial p_{0,i}(x,t)}{\partial t}+\frac{\partial p_{0,i}(x,t)}{\partial x}=-(a_1+a_2(x))p_{0,i}(x,t)+a_1p_{0,i-1}(x,t),\\
&\hspace{4cm}\; 2\le i\le m,\\
&\frac{\partial p_{1,1}(x,t)}{\partial t}+\frac{\partial p_{1,1}(x,t)}{\partial x}=-(a_1+b_2(x))p_{1,1}(x,t)\\
&\hspace{4cm}+\int^{\infty}_0p_4(x,u,t)b_1(u)\mathrm{d}u,\\
&\frac{\partial p_{1,i}(x,t)}{\partial t}+\frac{\partial p_{1,i}(x,t)}{\partial x}=-(a_1+b_2(x))p_{1,i}(x,t)\\
&\hspace{4cm}+a_1p_{1,i-1}(x,t),\; 2\le i\le m,\\
&\frac{\partial p_2(x,u,t)}{\partial t}+\frac{\partial p_2(x,u,t)}{\partial x}+\frac{\partial p_2(x,u,t)}{\partial u}\\
&\hspace{4cm}=-(a_2(x)+ b_1(u))p_2(x,u,t),\\
&\frac{\partial p_3(x,t)}{\partial t}+\frac{\partial p_3(x,t)}{\partial x}=-b_1(x)p_3(x,t)
+\int^{\infty}_xa_2(u)p_2(x,u,t)\mathrm{d}u,\\
&\frac{\partial p_4(x,u,t)}{\partial t}+\frac{\partial p_4(x,u,t)}{\partial x}=-b_1(x)p_4(x,u,t),\\
&p_{0,1}(0,t)=\int^{\infty}_0p_{1,1}(x,t)b_2(x)\mathrm{d}x+\delta(t),\\
&p_{0,i}(0,t)=\int^{\infty}_0p_{1,i}(x,t)b_2(x)\mathrm{d}x,\quad 2\le i\le m,\\
&p_{1,1}(0,t)=\int^{\infty}_0p_3(x,t)b_1(x)\mathrm{d}x+\int^{\infty}_0p_{0,1}(x,t)a_2(x)\mathrm{d}x,\\
&p_{1,i}(0,t)=\int^{\infty}_0p_{0,i}(x,t)a_2(x)\mathrm{d}x,\; 2\le i\le m,\\
&p_2(0,u,t)=a_1p_{0,m}(u,t),\\
&p_3(0,t)=0,\\
&p_4(0,u,t)=a_1p_{1,m}(u,t),\\
&p_{0,1}(x,0)=\delta(x),\; p_{0,i}(x,0)=0,\; 2\le i\le m,\\
&p_{1,i}(x,0)=0,\quad 0\le i\le m,\\
&p_2(x,u,0)=p_3(x,u,0)=p_4(x,u,0)=0.
\end{align*}
Where $(x,u,t)\in [0,\infty)\times[0,\infty)\times[0,\infty);$ $p_{0,i}(x,t)$ is the probability that at time $t$ two units are operating and the operating unit 1
is in phase $i,$ and the age of unit 2 is $x,\; i=1, 2, \cdots, m;$ $p_{1,i}(x,t)$ is the probability that at time $t$ unit 1 is operating, the operating phase is
$i$ and unit 2 is being repaired and the elapsed time of the failed unit 2 is $x,\; i=1, 2, \cdots, m;$  $p_2(x,u,t)$ is the probability that at time $t$ unit 2 is
operating and unit 1 is being repaired, the elapsed repair time of the failed unit 1 is $x,$ and the age of unit 2 is $u,;$ $p_3(x,t)$ is the probability that
at time $t$ unit 1 is being repaired and unit 2 is waiting for repair, and the elapsed repair time of the failed unit 1 is $x;$ $p_4(x,u,t)$ is the probability that
at time $t$ unit 1 is being repaired and unit 2 is in ``suspended repair", and the elapsed repair time of the failed unit 1 is $x$ and the elapsed time of the failed
unit 2 is $u;$ $a_1>0$ is a constant parameter of the Erlang distribution; $a_2(x)$ is the repair rate of unit 1; $b_1(x)$ is the life rate; $b_2(x)$ is the repair rate
of unit 2; $\delta(\cdot)$ is the Dirac-delta function.

Well-posedness of the above model and asymptotic behavior of its time-dependent solution have not been obtained until today.

In 1997,  Su \cite{Subaohe} established the following system of equations by using the supplementary variable technique:
\begin{align*}
\frac{\partial p_{0,0,k}(x,t)}{\partial t}+\frac{\partial p_{0,0,k}(x,t)}{\partial x}&=-(\lambda_{0,1}+\lambda_{1,0}+\alpha(x))p_{0,0,k}(x,t),\;k\ge 0,\\
\frac{\partial p_{0,1,0}(x,t)}{\partial t}+\frac{\partial p_{0,1,0}(x,t)}{\partial x}
&=-(\lambda_{1,0}+\lambda_{0,2}+\alpha(x))p_{0,1,0}(x,t)\\
&\quad+\sum^{\infty}_{k=0}\lambda_{0,1}p_{0,0,k}(x,t),\\
\frac{\partial p_{1,0,0}(x,t)}{\partial t}+\frac{\partial p_{1,0,0}(x,t)}{\partial x}
&=-(\lambda_{0,1}+\lambda_{2,0}+\alpha(x))p_{1,0,0}(x,t)\\
&\quad+\sum^{\infty}_{k=0}\lambda_{1,0}p_{0,0,k}(x,t),\\
\frac{\partial p_{0,1,1}(x,t)}{\partial t}+\frac{\partial p_{0,1,1}(x,t)}{\partial x}&=-\mu_{0,1}(x)p_{0,1,1}(x,t),\\
\frac{\partial p_{1,0,1}(x,t)}{\partial t}+\frac{\partial p_{1,0,1}(x,t)}{\partial x}&=-\mu_{1,0}(x)p_{1,0,1}(x,t),\\
\frac{\partial p_{0,2,0}(x,t)}{\partial t}+\frac{\partial p_{0,2,0}(x,t)}{\partial x}&=-\mu_{0,2}(x)p_{0,2,0}(x,t),\\
\frac{\partial p_{2,0,0}(x,t)}{\partial t}+\frac{\partial p_{2,0,0}(x,t)}{\partial x}&=-\mu_{2,0}(x)p_{2,0,0}(x,t),\\
\frac{\partial p_{1,1,0}(x,t)}{\partial t}+\frac{\partial p_{1,1,0}(x,t)}{\partial x}
&=-(\lambda_{0,2}+\lambda_{2,0}+\alpha(x))p_{1,1,0}(x,t)\\
&\quad+\lambda_{1,0}p_{0,1,0}(x,t)+\lambda_{0,1}p_{1,0,0}(x,t),\\
\frac{\partial p_{1,1,1}(x,t)}{\partial t}+\frac{\partial p_{1,1,1}(x,t)}{\partial x}&=-\mu_{1,1}(x)p_{1,1,1}(x,t),\\
\frac{\partial p_{1,2,0}(x,t)}{\partial t}+\frac{\partial p_{1,2,0}(x,t)}{\partial x}&=-\mu_{1,2}(x)p_{1,2,0}(x,t),\\
\frac{\partial p_{2,1,0}(x,t)}{\partial t}+\frac{\partial p_{2,1,0}(x,t)}{\partial x}&=-\mu_{2,1}(x)p_{2,1,0}(x,t),\\
p_{0,0,0}(0,t)&=\delta(t)+\int^{\infty}_0\mu_{0,1}(x)p_{0,1,1}(x,t)\mathrm{d}x\\
&\quad+\int^{\infty}_0\mu_{1,0}(x)p_{1,0,1}(x,t)\mathrm{d}x\\
&\quad+\int^{\infty}_0\mu_{0,2}(x)p_{0,2,0}(x,t)\mathrm{d}x\\
&\quad+\int^{\infty}_0\mu_{2,0}(x)p_{2,0,0}(x,t)\mathrm{d}x\\
&\quad+\int^{\infty}_0\mu_{1,2}(x)p_{1,2,0}(x,t)\mathrm{d}x\\
&\quad+\int^{\infty}_0\mu_{2,1}(x)p_{2,1,0}(x,t)\mathrm{d}x\\
&\quad+\int^{\infty}_0\mu_{1,1}(x)p_{1,1,1}(x,t)\mathrm{d}x,\\
p_{0,0,k}(0,t)&=\int^{\infty}_0\alpha(x)p_{0,0,k-1}(x,t)\mathrm{d}x,\; k\ge 1,\\
p_{0,1,0}(0,t)&=0,\\
p_{1,0,0}(0,t)&=0,\\
p_{0,1,1}(0,t)&=\int^{\infty}_0\alpha(x)p_{0,1,0}(x,t)\mathrm{d}x,\\
p_{1,0,1}(0,t)&=\int^{\infty}_0\alpha(x)p_{1,0,0}(x,t)\mathrm{d}x,\\
p_{0,2,0}(0,t)&=\int^{\infty}_0\lambda_{0,2}p_{0,1,0}(x,t)\mathrm{d}x,\\
p_{2,0,0}(0,t)&=\int^{\infty}_0\lambda_{2,0}p_{1,0,0}(x,t)\mathrm{d}x,\\
p_{1,1,0}(0,t)&=0,\\
p_{1,1,1}(0,t)&=\int^{\infty}_0\alpha(x)p_{1,1,0}(x,t)\mathrm{d}x,\\
p_{1,2,0}(0,t)&=\int^{\infty}_0\lambda_{0,2}p_{1,1,0}(x,t)\mathrm{d}x,\\
p_{2,1,0}(0,t)&=\int^{\infty}_0\lambda_{2,0}p_{1,1,0}(x,t)\mathrm{d}x,\\
p_{0,0,0}(x,0)&=\delta(x),\; p_{0,0,k}(x,0)=0,\quad k\ge 1,\\
p_{1,0,0}(x,0)&=0,\; p_{0,1,0}(x,0)=0,\\
p_{0,1,1}(x,0)&=0,\; p_{1,0,1}(x,0)=0,\\
p_{0,2,0}(x,0)&=0,\; p_{2,0,0}(x,0)=0,\\
p_{1,1,0}(x,0)&=0,\; p_{1,1,1}(x,0)=0,\\
p_{1,2,0}(x,0)&=0,\; p_{2,1,0}(x,0)=0.
\end{align*}
Where $\lambda_{1,0}$ is the failure rate of unit 1 from normal to partial failure; $\lambda_{2,0}$ is the failure rate of unit 1 from partial failure to total failure;
 $\lambda_{0,1}$ is the failure rate of unit 2 from normal to partial failure; $\lambda_{0,2}$ is the failure rate of unite 2 from partial failure to total failure;
  $\alpha(x)$ is  the check rate satisfying $\alpha(x)\ge 0,\; \int^{\infty}_0\alpha(x)\mathrm{d}x=\infty;$ $\mu_{i,0}(x)\; (i=1,2)$ is repair rate of of unit 1
  satisfying  $\mu_{i,0}(x)\ge 0,\; \int^{\infty}_0\mu_{i,0}(x)\mathrm{d}x=\infty;$ $\mu_{0,i}(x)\;(i=1,2)$ is repair time of unit 2 satisfying
$\mu_{0,i}(x)\ge 0,\; \int^{\infty}_0\mu_{0,i}(x)\mathrm{d}x=\infty;$ $\mu_{i,j}(x)\; (i,j=1,2)$ is repair rate of both the units satisfying
$\mu_{i,j}(x)\ge 0,\; \int^{\infty}_0\mu_{i,j}(x)\mathrm{d}x=\infty;$ $p_{0,0,k}(x,t)\; (k\ge 0)$ is the probability that unit 1 is normal, unit 2 is normal, the system
has been checked $k$ times and has an elapsed time of $x;$ $p_{1,0,0}(x,t)$ is the probability that unit 1 is partially failed, unit 2 is normal, the system has not been checked and has an
elapsed time of $x;$ $p_{0,1,0}(x,t)$
is the probability that unit 1 is normal, unit 2 is partially failed, the system has not been checked and has an elapsed time of $x;$ $p_{1,1,0}(x,t)$ is the probability that unit 1 is partially
 failed, unit 2 is partially failed, the system has not been checked and has an elapsed time of $x;$ $p_{0,1,1}(x,t)$ is the probability that unit 1 is normal,
 unit 2 is partially failed, the system has been checked one time and has an elapsed repair time of $x;$ $p_{1,0,1}(x,t)$ is the probability that unit 1 is partially failed, unit 2 is normal, the system has been checked one time and has an elapsed
repair time of $x;$ $p_{0,2,0}(x,t)$ is the
probability that unit 1 is normal, unit 2 is total failure and the system has not been checked; $p_{2,0,0}(x,t)$ is the probability that unit 1 is total failure, unit 2
is normal and the system has not been checked; $p_{1,1,1}(x,t)$ is the probability that unit 1 is partially failed, unit 2 is partially failed,
 the system has been checked one time and has an elapsed repair time of $x$; $p_{1,2,0}(x,t)$ is the probability that unit 1 is partially failed, unit 2 is totally failed, the system has not been checked and has an elapsed repair time of $x$;
 $p_{2,1,0}(x,t)$ is the probability that unit 1 is totally failed, unit 2 is partially failed, the system has not been checked and has an elapsed repair time of $x.$.

Well-posedness of the system of the above equations and asymptotic behavior of its time-dependent solution have not been obtained until now.

In 2006, Yue et al. \cite{YueZhuQin} established the following two models:
\begin{align*}
&\frac{\partial Q_0(x,t)}{\partial t}+\frac{\partial Q_0(x,t)}{\partial x}=-(2\lambda+\alpha(x))Q_0(x,t),\\
&\frac{\partial Q_1(x,t)}{\partial t}+\frac{\partial Q_1(x,t)}{\partial x}=-(2\lambda+\alpha(x))Q_1(x,t)+2\lambda Q_0(x,t),\\
&\frac{\partial Q_2(x,t)}{\partial t}+\frac{\partial Q_2(x,t)}{\partial x}=-(2\lambda+\alpha(x))Q_2(x,t)+2\lambda Q_1(x,t),\\
&\frac{\partial Q_3(x,t)}{\partial t}+\frac{\partial Q_3(x,t)}{\partial x}=-(2\lambda+\mu(x))Q_3(x,t),\\
&\frac{\partial Q_4(x,t)}{\partial t}+\frac{\partial Q_4(x,t)}{\partial x}=-(\lambda+\mu(x))Q_4(x,t)+2\lambda Q_3(x,t),\\
&Q_0(0,t)=\int^{\infty}_0\alpha(x)Q_0(x,t)\mathrm{d}x+\int^{\infty}_0\mu(x)Q_3(x,t)\mathrm{d}x+\delta(t),\\
&Q_1(0,t)=Q_2(0,t)=0,\\
&Q_3(0,t)=\int^{\infty}_0\alpha(x)Q_1(x,t)\mathrm{d}x+\int^{\infty}_0\mu(x)Q_4(x,t)\mathrm{d}x,\\
&Q_4(0,t)=\int^{\infty}_0\alpha(x)Q_2(x,t)\mathrm{d}x,\\
&Q_0(x,0)=\delta(x),\; Q_i(x,0)=0,\; i=1,2,3,4.
\end{align*}
Here $(x,t)\in [x,\infty)\times[0,\infty);$ $Q_i(x,t)\; (i=0,1,2)$ is the probability that at time $t$ the system is normal and the elapsed vacation time of
repairman is $x;$ $Q_i(x,t)\; (i=3,4)$ is the probability that at time $t$ the system is down and the elapsed repair time of
the piece is $x;$ $\lambda$ is the working rate, $\mu(x)$ is the repair rate, $\alpha(x)$ is the vacation rate of the repairman, $\delta(\cdot)$ is the Dirac-delta
function.
\begin{align*}
&\frac{\partial P_0(x,t)}{\partial t}+\frac{\partial P_0(x,t)}{\partial x}=-(2\lambda+\alpha(x))P_0(x,t),\\
&\frac{\partial P_1(x,t)}{\partial t}+\frac{\partial P_1(x,t)}{\partial x}=-(2\lambda+\alpha(x))P_1(x,t)+2\lambda P_0(x,t),\\
&\frac{\partial P_2(x,t)}{\partial t}+\frac{\partial P_2(x,t)}{\partial x}=-(\lambda+\alpha(x))P_2(x,t)+2\lambda P_1(x,t),\\
&\frac{\partial P_3(x,t)}{\partial t}+\frac{\partial P_3(x,t)}{\partial x}=-(2\lambda+\mu(x))P_3(x,t),\\
&\frac{\partial P_4(x,t)}{\partial t}+\frac{\partial P_4(x,t)}{\partial x}=-(\lambda+\mu(x))P_4(x,t)+2\lambda P_3(x,t),\\
&\frac{\partial P_5(x,t)}{\partial t}+\frac{\partial P_5(x,t)}{\partial x}=-\alpha(x))P_5(x,t)+\lambda P_2(x,t),\\
&\frac{\partial P_6(x,t)}{\partial t}+\frac{\partial P_6(x,t)}{\partial x}=-\mu(x)P_6(x,t)+\lambda P_4(x,t),\\
&P_0(0,t)=\int^{\infty}_0\alpha(x)P_0(x,t)\mathrm{d}x+\int^{\infty}_0\mu(x)P_3(x,t)\mathrm{d}x+\delta(t),\\
&P_1(0,t)=P_2(0,t)=P_5(0,t)=0,\\
&P_3(0,t)=\int^{\infty}_0\alpha(x)P_1(x,t)\mathrm{d}x+\int^{\infty}_0\mu(x)P_4(x,t)\mathrm{d}x,\\
&P_4(0,t)=\int^{\infty}_0\alpha(x)P_2(x,t)\mathrm{d}x+\int^{\infty}_0\mu(x)P_6(x,t)\mathrm{d}x,\\
&P_6(0,t)=\int^{\infty}_0\alpha(x)P_5(x,t)\mathrm{d}x,\\
&P_0(x,0)=\delta(x),\; P_i(x,0)=0,\quad i=1,2,3,4,5,6.
\end{align*}
Here $(x,t)\in [x,\infty)\times[0,\infty);$ $P_i(x,t)\; (i=0,1,2,5)$ is the probability that at time $t$ the system is normal and the elapsed vacation time of
repairman is $x;$ $P_i(x,t)\; (i=3,4,6)$ is the probability that at time $t$ the system is down and the elapsed repair time of
the piece is $x;$ $\lambda$ is the working rate, $\mu(x)$ is the repair rate, $\alpha(x)$ is the vacation rate of the repairman, $\delta(\cdot)$ is the Dirac-delta
function.

When $\alpha(x)$ and $\mu(x)$ are constants, it is the same as theorem \ref{geni}, one may obtain the explicit expressions of the time-dependent solutions of the above two models. When $\alpha(x)$ and $\mu(x)$ are non-constants, well-posendess and asymptotic behavior of the time-dependent solutions of the above two models have not been obtained until today.

\textbf{Acknowledgments:} This work is supported by the National Natural Science Foundation of China (No: 11961062).
The author would like to thank professor Frank Neubrander for his valuable suggestions.

\bibliographystyle{plain}

\end{document}